\documentclass[9pt, a4paper]{amsart}

\usepackage[left=2cm, right=2cm, top=2cm, bottom=2cm]{geometry}
\usepackage{amsaddr}

\usepackage[utf8]{inputenc}
\usepackage[english]{babel} %tavutus
\usepackage{amsthm}
\usepackage{amsmath}
\usepackage{mathtools}
\usepackage{amssymb}
\usepackage[all]{xy}
\usepackage{tikz}
\usepackage{enumerate} % (i) (ii) (iii) etc. to enumerate 
\usepackage{hyperref}
\usepackage{mathrsfs}
\usepackage{pgfplots}
\usepackage{tikz-3dplot}

\theoremstyle{plain}
\newtheorem{theorem}{Theorem}[section]
\theoremstyle{definition}
\newtheorem {definition}{Definition}[section]
\theoremstyle{remark}
\newtheorem{remark}{Remark}[section]
\theoremstyle{plain}

\theoremstyle{plain}

\theoremstyle{plain}
\newtheorem{corollary}{Corollary}[section]
\theoremstyle{plain}

\theoremstyle{plain}
\newtheorem{lemma}{Lemma}[section]

\def\R{\mathbb{R}}

\def\N{\mathbb{N}}

\def\mo{p}

\title{Microlocal Analysis of Waves Across the Event Horizon  of an Extremal Rotating Black Hole}
\author{Antti Kujanp\"a\"a}
\address{
Jockey Club Institute for Advanced Study,  \\
Lo Ka Chung Building, Lee Shau Kee Campus, The Hong Kong University of Science and Technology, \\
Clear Water Bay, Kowloon, 
HKSAR, China, \ 
email: iaskujanpaa@ust.hk, \\
LUT University, School of Engineering Science, PO Box 20, FI-53851
Lappeenranta, Finland, \\
email: antti.kujanpaa@lut.fi}
\date{\today}

\makeindex

\begin{document}

 \maketitle
 \begin{abstract}
Static black holes 
contain regions of spacetime which not even light can escape from.
In the centre of mass frame, these blocks are separated from each other by event horizons. 
Unlike pointlike particles, fields can spread and interact non-causally across the horizons. 
The microlocal theory   
of this is somewhat incomplete, however.    
For instance, the theory of real principal type operators does not apply on the horizon.
In this article, we address this issue for the extremal rotating black hole. 
Namely, we show that null covectors on the horizon of the extremal Kerr spacetime form an involutive double characteristic manifold 
and then extend the construction of parametrix across the event horizon. 
This provides a mathematical basis for the asymptotic oscillatory solutions in the region. Such approximations are central in quantum mechanics. 
In contrast to the real principal case, solutions near the horizon admit two channels for propagation of singularities. 
Indeed, there is additional propagation along the double characteristic variety.

\end{abstract}

\section{Introduction}

%(MORE STUFF TO INTRO)
%\\

Free fall of pointlike particles in general relativity is described by causal geodesics. 
Near a black hole, these trajectories tend to curve towards the singularity behind the event horizon. 
Paradoxically, an external observer never witnesses these particles reaching the interior. 
%For an observer outside the black hole, the particles never reach the interior, however.  
Instead, he (or she) sees them asymptotically piling up on the outer horizon, eventually forming thin layers (holograms) of information. 
%Moving the initial value (4-position-4-momentum) of a particle in the phase space closer to the horizon seems to converge to a worldline that is trapped in the horizon. 
%At the horizon causal trajectories converge into curves along the horizon. 
%\footnote{These curves are not ``frozen'' in time if the black hole is rotating.}
From this perspective, each causal trajectory in the spacetime lies strictly in the exterior, interior, or on the horizon of the black hole and there is no overlapping between these three classes. \footnote{The horizon may be crossed in the particle's own frame of reference. For an external observer, the crossing would take infinitely long time. Notice that this is not due to degeneracy of local coordinates used but rather a more fundamental issue in representing the two perspectives simultaneously in a single smooth framework.
Indeed, the ``coordinate'' transformation from the particle's own frame (e.g. Kerr star coordinates) to the external observer's frame (e.g. Boyer-Lindquist coordinates) is singular. 
}
This does not apply to fields  (quantum, electromagnetic, etc.) however which spread to the surrounding space. 
%In fact, the more localised the field is in the spacetime the less accuracy (cf. Heisenberg's uncertainty principle) it has in the momenta. 
%Indeed, even a delta field has 
%A field (quantum, electromagnetic, etc.) on the horizon lives simultaneously on both sides of it. 
For example, the quantum description of a spinless particle as an excitation in a Klein-Gordon field allows the probability density to distribute spatially across the horizon. Notice that the volume form $dV_g$ remains regular and non-zero on the horizon, at least for the known black holes. That is; the form does not ``see'' the horizon. Hence, the probability density as a measure $u dV_g$ (Hodge dual) has the same regularity and zeros as $u$.\footnote{Indeed, the volume form cannot erase the density or create irregularitites. That is; the Hodge star $*: u\mapsto u dV_g$ is elliptic.} 
%\footnote{In particular, a probability density as a scalar $u$ has the same regularity and roots near the horizon as its Hodge dual $*u = u dV_g$.}
The aforementioned wave functions are fundamental in studying quantum phenomena near the horizon. For instance, Hawking radiation is understood as pair production in this region. 
Yet there is a fundamental mathematical problem: 
%the standard microlocal theory of PDEs does not apply to fields on the horizon. 
The standard microlocal methods do not apply to fields on the horizon. 
This implies that the asymptotic oscillatory approximations used in physics to construct solutions (e.g. propagators) cannot be extended to the horizon in the usual sense. 
%(approximate Green's functions, fundamental solutions, propagators) 
%have no mathematical basis. 
%For the non-extremal Kerr and Schwarzschild black hole, this is due to characteristic radial points\footnote{I.e. 4-position-4-momentum pairs $(q,\mo)$ satisfying $s (q,\mo) = 0$ and $d_\mo s (q,\mo) = 0 \neq d_q s (q,\mo)$, where $s $ stands for the principal symbol of the operator}. 
{\color{black} For a sub-extremal black hole, one may resolve the issue by applying techniques developed in \cite{2011AnHP...12.1349W}, \cite{articledy}, \cite{Dyatlov2013ResonancePA}, \cite{Vasy2010MicrolocalAO} \cite{hintz2021normally}, \cite{jia2022propagation}, \cite{jia2023applications}.}
The extremal case is fundamentally different, however. 
The main point in this article is to show that the problem can be resolved (away from the normal of the horizon) for the extremal rotating black hole using the theory of operators with double characteristics. 
Indeed, we shall show that the radial points become double characteristics  ($ \sigma (q,\mo)  = 0 = d_{q,\mo} \sigma (q,\mo) $) at the extremal limit, and then apply the result from \cite{Uhlmann-Thesis} to solve  the Cauchy problem microlocally. Subsequently, the microlocal inverse of the operator is derived via Duhamel's principle.  The solutions admit two independent channels for propagation of singularities. Namely, the standard one along light cones in the real principal type regions, and an additional one on the horizon for double characteristics. Surprisingly, the extra propagation is not the limit of the standard one but an actual canonical relation with more structure. 
In particular, the dimension of the additional manifold is the same as for the standard one outside the event horizon.\footnote{The dimension of a canonical relation is half the dimension of the ambient space.} 
This is in line with the holographic principle \cite{susskind} (see also the AdS/CFT correspondence \cite{Maldacena}) which postulate a one-to-one correspondence between the exterior and the horizon. 
%Perhaps the most successful realisation of the principle is the AdS/CFT correspondence. 
%The microlocal properties of the wave on the horizon and outside it should reflect 
%This supports the hypothesis of the holographic principle.  
%Namely, that the collection of information (in our case the worldline for every individual null 4-momentum-4-position encoded in the canonical relation) in a time-reversable system outside the black hole can be encoded onto the event horizon. 
%The image in this mapping would be the additional canonical relation for double characteristics. 
%Indeed, the complete manifold of information (4-position-4-momentum) of photons can be encoded to the horizon. 
%These  phenomena give rise to the behaviour of classical  particles on the large scale.\footnote{Indeed, classical mechanics follows from quantum mechanics.} 
%The laws of motion are encoded into the canonical relations of the operator. 
%In particular, these 
%We shall compute explicitly how the wave front set evolves on the horizon. 

As indicated above, this article is about waves in the spacetime of a rotating black hole. We take the perspective of an external observer outside the black hole and focus on the extremal case of angular velocity $a = r_s/2 = G M_{BH} /c^2$.
This is a Kerr black hole with the minimum possible mass $M_{BH} = M_{BH}(a)$, or equivalently, the maximal speed of rotation $a= a(M_{BH})$.\footnote{Superextremal  ($a > G M_{BH} /c^2$) Kerr solutions do not have event horizons. They are widely considered as non-physical.} 
%Let $\square_g : u \mapsto  \frac{1}{\sqrt{- \det g}} \partial_\mu (\sqrt{- \det g} g^{\mu \nu} \partial_\nu u)  $ be the wave operator associated with the extremal Kerr metric $g$. 
%After rescaling the blowup we may 
%Thanks to the conformal invariance of our microlocal framework, we may extend the concept of waves to the horizon by rescaling the operator $\square_g$. 
We construct the microlocal parametrix for the Cauchy problem on a neighbourhood that  
crosses the event horizon of the black hole. As a corollary, the associated microlocal inverse\footnote{Confusingly, this is called parametrix too.} (approximate Green's function) of the operator is obtained. 
% propagating approximate inverse, a microlocal parametrix, of the (rescaled) wave operator on a neighbourhood that  
%crosses the event horizon of the black hole. 
The microlocal properties of waves on the horizon are determined. 
In particular, we compute explicitly how the wave front set evolves on the horizon. 
%More precisely, we show that null covectors the event horizon is an involutive double characteristic set visible in waves as a additional singularities orbiting along the horizon. 
%Such solutions are important in physics as quantum mechanics, electromagnetism, and even gravitation (e.g. linearized gravity under Lorentz gauge) 
%reduce to study of waves.
%\footnote{The principal part can be diagonalized, whereas the coupling in the lower order terms can be removed microlocally by applying elliptic operators.}
%For example, a quantum particle (pair) created on the event horizon (cf. Hawking pair) is a wave mode that spreads its probability distribution spatially into the interior and exterior of the black hole simultaneuously. 
%Notice that the interior and exterior of the black hole can be entangled even though no causal (i.e. at speed $\leq c$) exchange of information across the horizon takes place. 

%The study is based on the theory of operators with double characteristics, developed in \cite{Uhlmann-Thesis}, \cite{Conical-Refraction}. 
Let $\square_g : u \mapsto  \frac{1}{\sqrt{- \det g}} \partial_\mu (\sqrt{- \det g} g^{\mu \nu} \partial_\nu u)  $ be the wave operator associated with the extremal Kerr metric $g$. 
The geometric picture of waves (e.g. light) near the horizon can be described in the form $Pu = 0$, where $P = \Delta \square_g$ is an operator with smooth coefficients.\footnote{This extends to vector-valued waves too (e.g. light as a solution to Maxwell equations) as it is often possible to decouple a system $P^\mu_\nu u_\mu = 0$ microlocally  into a set of independent scalar waves by conjugating with elliptic operators (see \cite{Taylor-Book}). }
Here the element $\Delta = (r-r_s/2)^2$ is to balance out the blowup of the metric so that the smooth framework can be applied across the horizon. 
In this rescaled form, the event horizon appears (see Section \ref{lemmas}) as a set of involutive double characteristic points instead of a blowup. 
One way interpret this is that the spacetime admits motion in two different levels: 
On one hand, we have the standard real principal type regimes outside and inside the black hole. In these regions wave fronts (wave packages, particles) evolve along the characteristic manifold (null covector bundle) $\Sigma =  \{ (q,\mo) \in T^* M \setminus \{0\} :  \mathcal{H} (q,p)   = 0 \} $ 
%$\mathcal{H} (q,p) :=g^{\mu \nu }(q)  \mo_\mu \mo_\nu $, 
according to the canonical Hamiltonian flow 
\begin{equation}\label{flowa}
\begin{cases}
\dot{q} = \partial_\mo \mathcal{H},\\
\dot{\mo} = -\partial_q \mathcal{H}
\end{cases}
\end{equation}
 (see \cite{Hormander-FIO1,Hormander-FIO2}, \cite{Melrose-Uhlmann}) with the conserved quantity being mass, or strictly speaking, mass${}^{2}/2$ defined by $\mathcal{H} =  - \frac{1}{2}g^{\mu \nu }(q)  \mo_\mu \mo_\nu$.\footnote{The curves solving \eqref{flowa} in $\Sigma$ are considered as sets. The theory is conformally invariant: Multiplying $g$, and hence $\mathcal{H} $, with a positive function changes only the parametrisation of the curves.} 
 On the other hand,  $d \mathcal{H}  = 0$ holds for null covectors on the horizon which indicates absence of the flow \eqref{flowa} at such points. 
 Instead, the set of double characteristics, 
 \[
\{ (q,\mo) \in T^* M \setminus \{0\} :  \mathcal{H} (q,p)   = 0, \ d \mathcal{H} (q,p) = 0 \}  \subset \Sigma,
 \]
%which vanishes in 
%becomes 
becomes non-trivial there. 
%and 
%admits foliation into leaves which are analogous to the curves above. 
%In fact, the bundle of non-zero null covectors (light cones) on the horizon is exactly the space of double characteristics points, as shown in Lemma \ref{proporo}.  
This set admits a natural foliation into leaves which are analogous to the curves \eqref{flowa} outside the horizon. 
These second order characteristics reflect dynamics that arise from higher order conservation laws beyond the standard framework of Noether's theorem. 
%Indeed, operators with double characteristics admit additional propagation of singularities
It was shown in \cite{Uhlmann-Thesis}, \cite{Conical-Refraction} by Melrose and Uhlmann that operators with double characteristics indeed admit additional propagation of singularities
along the double characteristic variety, provided that certain conditions are satisfied. In fact, this is the mathematical explanation behind the conical refraction, an optical phenomenon in which light is refracted 
into a cone of rays when hitting a biaxial crystal along the optical axis. 
In the context of extremal black hole, the propagation of double characteristics describe kinematics on the event horizon. 
Notice that each wave lies simultaneously on both sides of the horizon. The intersection of the wave front and event horizon orbits along the surface at a constant angular velocity,  
whereas the remaining parts outside and inside follow the standard Hamiltonian dynamics (null geodesics) described in \eqref{flowa}.

%In the phase space, the 
%Our result also implies existence of a photon that rests at the pole of the horizon while spinning around its own axis. 
%Indeed, the wave operator $\square_g$ fails to be of real principal type on the horizon, indicating that the standard microlocal framework 
%(see \cite{Hormander-FIO1,Hormander-FIO2}, \cite{Melrose-Uhlmann}) does not apply. 

%\subsection{Toy Model: Schwarzschild Black Hole}
%As a simple example we consider the Schwarzschild (inverse) metric given by
%\[
%g =-  \left(1-\frac{r_s}{r} \right) dt^2 +  \left(1-\frac{r_s}{r} \right)^{-1}  dr^2 + r^2   d\theta^2 +r^2 \sin^2 \theta \  d\phi^2
%\]
%\[
%g^{\mu \nu} \partial_\mu \partial_\nu =-  \left(1-\frac{r_s}{r} \right)^{-1} \partial_t \otimes \partial_t +  \left(1-\frac{r_s}{r} \right)  \partial_r \otimes \partial_r  + r^{-2}   \partial_\theta \otimes \partial_\theta +r^{-2} \sin^{-2} \theta \ \partial_\phi \otimes \partial_\phi
%\]

\subsection{The Model}
Let $(M,g)$ be the extremal 
(angular velocity $a=r_s/2$) 
Kerr black hole. 
The associated volume  form $dV_g$ is given by
\[
dV_g = \sqrt{-\det g} \  dq^0 \wedge \cdots \wedge dq^3  
\]
In Boyer-Lindquist coordinates $(t,r,\theta,\phi)$,
\[
dV_g = \Sigma  \sin \theta  \ dt \wedge dr\wedge d\theta \wedge d\phi  
\]
where 
\[
\Sigma := r^2 + a^2 \cos^2 \theta =  r^2 + \frac{r_s^2}{4}  \cos^2 \theta.
\]
%We are interested in the extreme case, corresponding to the angular velocity $a=r_s/2$. 
%This gives 
%\[
%\Sigma = r^2 + \frac{r_s^2}{4}  \cos^2 \theta.
%\] 
Above $r_s$ stands for the Schwarzschild radius. It corresponds to the mass $M_{BH}$ of the black hole via $r_s = 2G M_{BH} /c^2$ where $G$ is the gravitational constant and $c$ is the speed of light. 
The ring singularity $\{ \Sigma= 0 \} = \{ \theta = \pi/2 , \ r = 0\}$ is where the curvature blows up and 
%we denote $\dot{M} = M \setminus \{ \Sigma= 0 \}$. 
we shall exclude it from our definition of the spacetime $M$. 

%We define the wave operator $\square_g u = \frac{1}{\sqrt{-\det g}} \partial_\mu(  \sqrt{-\det g } g^{\mu \nu } \partial_\nu u   ) $  to half densities $
%%u  \sqrt{dV_g} \in
%  C^\infty_c ( \Omega_{\frac{1}{2}} ; M)$ in the standard way. 
%\[
%\begin{split}
%\langle \square_g u \sqrt{dV_g}   , \varphi \sqrt{dV_g} \rangle &= \int_M  g( d u ,d \varphi) dV_g  , \quad \forall \varphi \in C_c^\infty(M). %= \sum_\alpha  \int_{U_\alpha}  \varphi \circ  \psi_\alpha   \partial_\mu(  \sqrt{-\det g } g^{\mu \nu } \partial_\nu u   )    \  dq^0 %\cdots dq^3
%\end{split} 
%\]
In the local canonical coordinates generated by Boyer-Lindquist coordinates, the principal part 
%($\langle \square_gu, \phi \rangle =- \sum_\alpha   \int \chi_j  u   g^{\mu \nu }(q) D_\mu D_\nu \phi (q)  \sqrt{-\det{g}}  dq $ ) 
%\[
$ - \sqrt{-\det{g}} \  g^{\mu \nu }(q) D_\mu D_\nu  $ 
% :  \phi \mapsto   \int \int   e^{i (q-\tilde{q}) \mo } g^{\mu \nu }(q) \mo_\mu \mo_\nu  \phi (\tilde{q})  d\tilde{q}  d\mo
%\]
%symbol $- g^{\mu \nu }(q) \mo_\mu \mo_\nu  \  dV_g \wedge  d\mo $ 
%\[
%%-g^{\mu \nu } \mo_\mu \mo_\nu  \sqrt{ \frac{dV_g \wedge dV_g \wedge  d\mo}{ dx \wedge -dx }}   =
%- \sqrt{-\det g} g^{\mu \nu } \mo_\mu \mo_\nu  
%%(\Sigma  \sin \theta   )   \sqrt{ d\mo }  , \quad dx  := \wedge_\mu dx^\mu , \quad d\mo := \wedge_\mu d\mo_\mu 
%\]
 of $\square_g $ for the Kerr black hole has 
%we have at $(q,\mo) \in T^*M$ the quadratic function
\[
%\begin{bmatrix} \mo & \mo \end{bmatrix}    \begin{bmatrix} g^{jk}(t,x) \end{bmatrix}  \begin{bmatrix} \omega \\ \mo \end{bmatrix}   
 g^{\mu \nu }(q) \mo_\mu \mo_\nu =  -  \frac{1}{c^2 \Delta} \left( r^2 + a^2 + \frac{r_s r   }{\Sigma}  a^2 \sin^2 \theta \right) \mo_t^2 - \frac{2a}{ c \Delta}  \frac{r_s r   }{ \Sigma} \mo_\phi \mo_t  
+\frac{\Delta}{\Sigma} \mo_r^2 + \frac{1}{\Sigma} \mo_\theta^2  + \frac{1}{\Delta  \sin^2 \theta} \left( 1- \frac{r_s r   }{\Sigma} \right)  \mo_\phi^2,
\]
(this formula holds also for non-extreme $a\geq 0$)
where 
\[
\Delta := r^2 -  r_s r  + a^2  = (r-r_s/2)^2. 
\] 
%which reduces in the extremal case to
%%Again, $a=r_s/2$ for the extreme Kerr black hole. Hence,
%\[
%\Delta= (r-r_s/2)^2. 
%\] 
%since $a=r_s/2$ for the extremal black hole. 
An immediate consequence of symmetry is that both the energy $\mo_t $ and angular momentum $\mo_\phi$ are constants of motion.
%and $ (t , r,\theta,\phi ; \mo^t,\mo^r,\mo^\theta,\mo^\phi) \in (\R \times (0,\infty) \times [0,\pi] \times [0,2\pi] ) \times \R^4 $ are the associated canonical coordinates. 
The axis $\mathscr{A } \subset M$ of rotation can be identified as the boundary $\{ \theta = n \pi \}$ of the coordinate chart and the covectors $T^* \mathscr{A}$ as the set $\{  \theta = n \pi, \ \mo_\phi = 0\}$. 
Define the horizon $\mathscr{H} \subset  M$ as the points with $r=  r_s/2$. 
This is the set where the metric $g$ blows up. Notice that the two horizons of the Kerr black hole coincide in the extremal case.

Below $\Psi^l(M)$ stands for properly supported pseudo-differential operators of order $l$ on $M$. 
Let $\kappa \in C^\infty(M)$. 
We are interested in solving $u$ from the equation 
\begin{equation}\label{skiq}
\left(\square_g  + \frac{\kappa}{\Delta} \right) u  =\frac{ f}{\Delta}  
\end{equation}
 microlocally near a horizon point $(q,\mo)  \in \dot{T}^* M  $, $q \in \mathscr{H}$. 
 A special yet important example is the Klein-Gordon equation given by $f=0$, $\kappa=m^2 \Delta $ which describes spinless particles with mass $m$. Further, setting $m = 0$ gives the wave equation. 
The source term $\frac{ f}{\Delta}$ usually arises from interactions between lower order perturbations in the non-linear (interacting) theory. Here $f$ can be singular, although some assumptions are required for the fraction $\frac{ f}{\Delta} $ to be well defined.  
For example, it suffices to take $f$ to be singular across a submanifold that meets the horizon cleanly (see \cite[Theorem 1.3.6.]{Duistermaat-Book}). 
%We require the source $f \in \mathcal{D}'( M)$ to satisfy the transversality condition 
%\[
%WF(f) \cap N^* \mathscr{H} = \emptyset
%\]
Since $\square_g + \frac{\kappa}{\Delta}$ blows up on the horizon, 
we rewrite the equation above in the conformally rescaled form
\begin{equation}\label{across-eq}
Pu =  f ,
\end{equation}
where 
\[
P := \Delta \square_g+ \kappa 
%= \square_{\Delta^{-1} g} +  \frac{2(r-r_s/2)^3}{\Sigma}  \partial_r  + \kappa 
\]
and consider it on the whole $M$ instead of working separately on the two blocks $M^\pm = \{ 0< \pm (r-  r_s/2) \} $ of $M \setminus \mathscr{H}$.
% with horizon as the boundary. 
This allows us extend the smooth PDE framework through the horizon.  
Indeed, every local representation of the operator $P$ has smooth coefficients 
and the solutions of the rescaled equation satisfy the original one in \eqref{skiq}.
Near $q \in M \setminus \mathscr{H}$, the operator $P$ is of real principal type. 
This case is well explored in \cite{Hormander-FIO1, Hormander-FIO2}, \cite{Melrose-Uhlmann}. 
\begin{remark}
Notice that, $P  \equiv \square_{\Delta^{-1} g}  \mod \Psi^1_{cl}(M)$. Indeed, we could rescale the metric $g$ instead. The lower order terms do not play a significant role in our analysis.
\end{remark}
\begin{remark}
Despite the blowup in the metric $g$, the volume form $dV_g$ is regular and nonvanishing. 
%Indeed, the form does not ``see'' the horizon. 
Hence, it can be used in the natural identification of smooth functions as $\frac{1}{2}$-densities. 
This identification is considered in the global description (See e.g. \cite[Chapter 11]{Grigis-Sjostrand}) of FIOs. 
\end{remark}

\subsection{The Results}\label{results}

Before stating the results, we briefly go through some key concepts related to them. 

\subsubsection{The Set $\Sigma_2$ of Double Charactertistics}
Let $\Sigma_2 \subset \dot{T}^* M$ be the submanifold  $\{ r=r_s/2, \ \mo_t +\Psi    = 0 \}$, $\Psi :=  \frac{g^{t\phi}}{g^{tt}} \mo_\phi  \in C^\infty(\dot{T}^* M )$ including the poles $ \{ r=r_s/2, \ \theta = m\pi,  \   \mo_\phi = 0, \ \mo_t    = 0 \    \}$ (NB. $\Psi|_{\mo_\phi= 0} = 0$).  
%and write it in coordinates $(t, \theta,\phi, \mo_r, \mo_\theta,\mo_\phi) $ using the parametrisation
%\begin{equation}\label{parewe}
%(t, \theta,\phi, \mo_r, \mo_\theta,\mo_\phi) \mapsto (t, r_s/2, \theta,\phi,  \Psi|_{r=r_s/2} (\theta,\mo_\phi) , \mo_r,\mo_\theta,\mo_\phi) .
%\end{equation}
We shall show (Lemma \ref{proporo}, Lemma \ref{involemma}) that $\Sigma_2$ is an involutive double characteristic set (see the definition in Section \ref{lemmas}) of the operator $P$. 
Moreover, the subprincipal symbol $c_P$ of $P$  vanishes on $\Sigma_2$  (see Lemma \ref{sublemma}) and the Hessian of the principal symbol of $P$ is (see Lemma \ref{hesslemma}) of rank 2 on $\Sigma_2$ outside the conormal bundle 
$N^* \mathscr{H} = \{ r= r_s/2, \ \mo_\theta = \mo_\phi = \mo_t = 0\} $ of the horizon $\mathscr{H} = \{ r = r_s/2\}$. 

%The manifold $\Lambda_2$ is the flowout canonical relation 
%set of operators $\mathcal{E}'(M) \to \mathcal{D}'(M)$ with a smooth Schwartz kernel. 
%Moreover, let $\mathcal{E}_\Gamma' (M)  := \{ u \in \mathcal{E}'(M) : WF(u) \subset \Gamma \}$. 

%Moreover, we denote by  $N^* \mathscr{H} \subset \dot{T}^* M $   the conormal bundle of the horizon. In coordinates, $N^* \mathscr{H} = \{ r = r_s/2, \ \mo_t = \mo_\theta = \mo_\phi = 0, \ \mo_r \neq 0 \}$. 
%

%The following reduction is carried out in Section \ref{mainsection}. 
%The techniques used were developed in \cite{Uhlmann-Thesis} (see also \cite{Conical-Refraction} for a more general case). 
%
\subsubsection{Microlocal Reduction}
The microlocal Cauchy problem for $P$ in given local canonical coordinates $q^j, p_j $  (see Egorov's theorem \cite[Theorem 10.1.]{Grigis-Sjostrand}) of $ T^* \R^4$ is about solving $ P u  = 0 $ with $ (  u ,  D_1 u )|_{q^1= 0} = (f_1, f_2)$ up to a smooth residual.  
Using the techniques in \cite{Uhlmann-Thesis} (see also \cite{Conical-Refraction} for a more general case),
we reduce  (see Section \ref{mainsection})  the Cauchy problem for $P$ 
microlocally near given $\lambda \in \dot{T}^* M \setminus N^* \mathscr{H}$ to the following model form 
\begin{align}
\tilde{\mathcal{P}}  \tilde{\mathcal{E} } \equiv 0,  \quad 
%\mod I^{-\infty} \\ 
 \tilde{\mathcal{E}} |_{x^0 = 0} \equiv \mathcal{I},  \label{caussimo}
 %\mod   I^{-\infty} 
\end{align}
where 
%
% the Cauchy problem for  operators with double characteristics is reduced microlocally into 
%$E \equiv E_1+E_2+E_3 \mod I^{-\infty}$ 
%a similar problem for the vector valued operator 
\begin{equation}
\tilde{\mathcal{P} } \equiv 
\begin{bmatrix} 
D_0-D_1 & 0 \\
0 & D_0
\end{bmatrix}
\mod  \mathcal{M}_{2\times 2} ( C^\infty( \R ;   \Psi^0 (\R^3)) ) , 
\end{equation}
 $\mathcal{I} :=  \begin{bmatrix} id & 0 \\ 0 & id  \end{bmatrix}$, and the underlying coordinates 
%The transformation applies
%which is obtained by applying elliptic pseudodifferential operator matrices of degree $0$.  
%This is written in new local canonical coordinates 
\begin{equation}\label{1Mia_coordinates}
(x,\xi) =  (\overbrace{x^0,x'}^{x}, \overbrace{\xi_0,\xi'}^\xi ) =
(x^0,\overbrace{x^1,x^2,x^3}^{x'},\xi_0,\overbrace{\xi_1,\xi_2,\xi_3}^{\xi'})  \in  U \subset \dot{T}^* \R^4 .
\end{equation}
of $\dot{T}^* M$ are obtained via local canonical transformation. 
%used here are determined as the completion (see \cite[Theorem 9.4.]{Grigis-Sjostrand}) of $x^0 = z^1$, $x^1 = z^1 - z_2$, $x^2=z^3$, $x^3=z^4$. 
%The system \eqref{caussimo} is the ultimate form of the microlocal Cauchy problem for $P$. 
Let $\mathcal{X} : \mathcal{D}'(M)  \to \mathcal{D}' (\R^4) $  
be an elliptic canonical transformation (see Egorov's theorem \cite[Theorem 10.1]{Grigis-Sjostrand}) to the coordinates \eqref{1Mia_coordinates} in the reduction above. We may assume that $\lambda$ is taken to an element in $ \dot{T}^*_0 \R^4$. Moreover, we identify $M$ as a subset of the Boyer-Lindquist coordinate chart with boundary, as before. 
The operator $\mathcal{X} $ is then a FIO in $I^0_{cl}(\Lambda_{\mathcal{X} })$, where the canonical relation $\Lambda_{\mathcal{X} }$ is the graph of the coordinate transformation. 
Let $\mathcal{Y} \in I^0_{cl}(\Lambda_{\mathcal{Y} }) $ be the microlocal inverse of $\mathcal{X} $. 
%The microlocal parametrix $E$ satisfying $P E \equiv 0 \mod C^\infty$ together with   $( \rho_1 E  ,  \rho_2 E  )   \equiv (I,I) $, 
The Cauchy data $u \mapsto (u|_{x^0=0}, D_0 u|_{x^0 = 0})$ translates in $M$ to $v \mapsto  (\rho_1 v ,\rho_2 v)$, where  $\rho_1 := \iota^*  \mathcal{X}  $, $\rho_2   :=  \iota^*  D_0 \mathcal{X} $ and $\iota : \R^3 \hookrightarrow \R^4$, $\iota(x') \mapsto (0,x')$. 
% is  $E(f_1,f_2)     \equiv \mathcal{A} \tilde{\mathcal{E} }_{22}  f_1 +  \mathcal{B}   \tilde{\mathcal{E} }_{21} f_2$, where $\mathcal{A}$ and $\mathcal{B}$ are elliptic.  
%used here correspond to the previous ones
% (e.g. the canonical pairs $q^\mu,p_\nu$ generated by Boyer-Lindquist coordinates $q=(t,r,\theta,\phi)$) 
%via $x^0 = z^1$, $x^1 = 

\subsubsection{The Canonical Relations}
A canonical relation is by definition a Lagrangian submanifold in the twisted product $(X \times Y, \sigma_X \oplus -\sigma_Y) $ of two symplectic spaces $(X, \sigma_X)$ and   $(Y, \sigma_Y)$. 
We consider it as a bundle equipped with the cartesian projection $(x,y) \mapsto y$  to the right. 
%A fiber of a point $y$ is the preimage $\pi^{-1} (y)$ of it in this projection. 
Let $\Lambda_{D_0}$, $\Lambda_{D_1}$ and $\Lambda_{D_0 - D_1}$ be the characteristic flowout canonical relations in $U \times  U$ 
%$U_0 := \{ (x,\xi)  \in U : x^0 = 0\}$, 
generated by $\xi_0$, $\xi_1$, and $\xi_0-\xi_1$, respectively.
That is; the pairs of points in $\{ \mathcal{H} = 0\}$ that are connected to each other via the flow \eqref{flowa}, where $\mathcal{H}$ is the generating function. 
Moreover, we define the canonical relations $C_{D_0}$, $C_{D_1}$ and $C_{D_0 - D_1}$ as the images of these manifolds in the projection $\text{pr}: (x,\xi, \tilde{x}, \tilde\xi ) \mapsto (x,\xi, \tilde{x}', \tilde\xi')$.  
%
%
%MUUTOS
The diagonal $\Lambda_0 = \{ (\lambda, \lambda) : \lambda \in U\}$ and $\text{pr} \Lambda_0$ divide the canonical relations $\Lambda_{(\cdot)}$, $C_{(\cdot)}$   transversally into forwards and backwards propagating halves 
% to forwards ($+$) and backwards ($-$) propagating halves 
$\Lambda_{(\cdot)}^\pm $, $C_{(\cdot)}^\pm$, respectively. 
We show in Section \ref{mainsection} that the fibres of the canonical relations $ \Lambda_{\mathcal{Y}} \circ \Lambda_{D_1} \circ C_{D_0}$ and $  \Lambda_{\mathcal{Y}} \circ \Lambda_{D_1} \circ C_{D_0-D_1}$ are parametrised by
% \begin{equation}\label{324022}
% \{ ( \Omega (s_1,s_2, \lambda),   \lambda ) :   \lambda \in \Sigma_2  \} \subset \Sigma_2 \times \rho_1\Sigma_2,  
% \end{equation}
%  where
\begin{equation}\label{cano2}
% \Omega (s_1,s_2 ; t, \theta,\phi, \mo_r, \mo_\theta,\mo_\phi ) := 
 (s_1,s_2) \mapsto   \left(t+ s_1,\frac{r_s}{2}, \theta , \phi +  \frac{c   }{r_s  } s_1   \ ; -\frac{c  }{ r_s } \mo_\phi , \mo_r (s_1,s_2) , \mo_\theta, \mo_\phi\right)  
\end{equation}
where $s_1,s_2$ are free real parameters and the $r$-momentum $\mo_r (s_1,s_2)= \mo_r (s_1,s_2, t, \theta,\phi, \mo_r, \mo_\theta,\mo_\phi )$ is some smooth function. 
Moreover, away from these relations, the manifolds $\Lambda_{\mathcal{Y}} \circ C_{D_0}$ and $ \Lambda_{\mathcal{Y}} \circ C_{D_0 - D_1}$ coincide with the characteristic flowout canonical relation $C_{\square_g} $ 
%$\Sigma_1 := \{  g^{\mu \nu} \xi_\mu \xi_\nu = 0 \} \setminus \{ r = r_s/2\}$ 
related to the wave operator $\square_g$ in the real principal type regions, i.e. away from the horizon. 
The analogous statement holds for $ \Lambda_{\mathcal{Y}} \circ \Lambda_{D_0}$ and $\Lambda_{\mathcal{Y}} \circ \Lambda_{D_0 - D_1}$ away from $\Lambda_{\mathcal{Y}} \circ \Lambda_{D_1} \circ ( \Lambda_{D_0}  \cup  \Lambda_{D_0 - D_1})$.

%We define $C_I = C_{D_0} \cup C_{D_0 - D_1}$ as a manifold with boundary $\partial C_I= C_{D_0} \cap C_{D_0 - D_1}$ and 
%$C_{II} = \Lambda_{D_1}  $
%a local homogeneous canonical transformation. 
% We refer to Section \ref{mainsection} for more details. 
Below we indicate by $I^{-\infty}$ the class of operators with a smooth Schwartz kernel. 
The main theorem regarding the Cauchy problem is as follows. See Section \ref{mainsection} for a proof, including the details of the reduction above. 

\subsubsection{The Main Theorem}

\begin{theorem}[Parametrix for the Cauchy Problem]\label{mainpropositheo}\label{main1}
Let $\lambda \in \dot{T}^* M  \setminus   N^* \mathscr{H}$.  
There exists a continuous operator $E =E_\lambda : \mathcal{E}' (\R^3 ) \times \mathcal{E}' (\R^3 ) \to \mathcal{D}'(M) $ such that, microlocally near $\lambda$, 
%\[
%WF'(E) \subset C_1 \cup C_2, 
%\]
%and, microlocally near $\lambda$, 
%$E=E_1+E_2$ with $E_1 \in I(C_1)$  
%and $WF' (E_2) \subset C_2 \cup C_1 $. 
\begin{align}
%&WF'(E) \subset C_1 \cup C_2, \\
& P E  \equiv 0 \mod I^{-\infty} , \\
&(\rho_1E ,\rho_2E)   \equiv (id,id), 
%&WF' (E) \subset C_2 \cup C_1, 
\end{align}
and
\[
  E   \in I^{{\color{black}1/4,-1/2}} ( \Lambda_{\mathcal{Y} } \circ   C_{D_0} ,   \Lambda_{\mathcal{Y} } \circ \Lambda_{D_1} \circ C_{D_0} ) + I^{{\color{black}1/4,-1/2}}( \Lambda_{\mathcal{Y} } \circ C_{D_0- D_1} ,  \Lambda_{\mathcal{Y} } \circ \Lambda_{D_1} \circ C_{D_0-D_1} ) ,
\]
that is, 
 in the canonical coordinates \eqref{1Mia_coordinates}, 
\[
\mathcal{X}  E   \in I^{{\color{black}1/4,-1/2}} (   C_{D_0} ,    \Lambda_{D_1} \circ C_{D_0} ) + I^{{\color{black}1/4,-1/2}} ( C_{D_0- D_1} ,   \Lambda_{D_1} \circ C_{D_0-D_1} ) 
\]
where $\mathcal{X},\mathcal{Y}$ and $(\rho_1,\rho_2) = ( \iota^* \mathcal{X} ,   \iota^* D_0 \mathcal{X} )  $ are as described above. 
The canonical relations $ \Lambda_{\mathcal{Y} } \circ \Lambda_{D_1} \circ C_{D_0}$ and $  \Lambda_{\mathcal{Y} } \circ \Lambda_{D_1} \circ C_{D_0-D_1}$  above have fibres of the form \eqref{cano2} and
the operator $E$ is microlocally unique modulo $I^{-\infty}$. 
%Away from $C_{D_1}^- \circ (C_{D_0} \cup C_{D_0-D_1})$, the relation $C_{D_0} \cup  C_{D_0 -D_1}$ coincides with the standard characteristic flowout $C_{\square_g} $ for $\square_g$ outside the horizon. 
%In particular, 
%\[
%WF'(E) \subset (C_{D_0}  \ \cup \  C_{D_0- D_1} ) \ \cup  \  C_{D_1}^- \circ (C_{D_0}  \ \cup \  C_{D_0- D_1} ). 
%\]
\end{theorem} 
\begin{remark}
It follows from the proof of the theorem above that 
$ \mathcal{X}E(f_1,f_2) = 
%\mathcal{A} 
E^1 f_1 + 
%\mathcal{B}
E^2 f_2$ for
\[
E^j  \equiv E_1^j + E_2^j + E_3^j, \quad j=1,2,
\]
where 
%$ \mathcal{A}$, $ \mathcal{B}$ are elliptic operators and 
$E_k^j$, $k=1,2,3$ have Schwartz kernels of the form
\begin{align}
E_1^j (x,y') & =  \int_{\R^3} e^{i(x'-y') \cdot  \zeta'} a^j(x,\zeta') d\zeta' \in I^{{\color{black}-1/4}}(C_{D_0}), \quad a^j \in S^0 (\R^4 \times \R^3 ) \label{cicada1}  \\ 
E_2^j  (x,y') & = \int_{\R^3} e^{i x^0 \zeta_1+i(x'-y')  \cdot  \zeta' } b^j(x,\zeta') d\zeta' \in I^{{\color{black}-1/4}}(C_{D_0-D_1}), \quad b^j \in S^0 (\R^4 \times \R^3 ) \label{cicada2} \\ 
E_3^j (x,y') & = \int_{-x^0}^{x^0} \int_{\R^3} e^{i\frac{r+x^0}{2}\zeta_1 + i(x'-y')\cdot \zeta'   } c^j(x,\zeta') d\zeta' dr,  \\
 &\in   I^{ {\color{black}1/4 ,-1/2} } ( C_{D_0} , \Lambda_{D_1} \circ C_{D_0} ) + I^{{\color{black}1/4,-1/2}} ( C_{D_0- D_1} , \Lambda_{D_1} \circ C_{D_0-D_1} ) , \quad c^j \in S^1 (\R^4 \times \R^3 ) \label{cicada3} , \quad j=1,2.
\end{align} 
{\color{black}Here we denote $\zeta' =( \zeta_1,\zeta_2,\zeta_3) \in \R^3$. }
%modulo $C^\infty (\R^4 \times \R^3 )  $. 
%The operator $E$ is microlocally unique modulo $I^{-\infty}$. 
%and $\Upsilon \in I (C_1) +  I (C_2) $.  
\end{remark}
Applying  Duhamel's principle (see \cite{Uhlmann-Thesis} for details) yields the following:
\begin{corollary}[Approximate Inverse of $P$]
%Let $P \equiv \square_{\Delta^{-1} g} \mod \Psi^{1}_{cl} (\dot{M})$. 
%Let $\Gamma $ be a closed conic neighbourhood in $\dot{T}^* M$ such that $ \Gamma \cap \{ \mo_t = 0 \} = \emptyset $.
Let $\lambda \in \dot{T}^*M    \setminus N^* \mathscr{H} $. %Let $\Lambda_0 = \{ ( \lambda , \lambda) : \lambda \in \dot{T}^* \R^4\}$ 
There exist continuous operators $G^\pm  : C^\infty_c (M)\to \mathcal{D}'(M) $\footnote{Forwards (+) / backwards (-) propagating.} such that, microlocally near $\lambda$, 
\[
P G^\pm  \equiv  I  \mod I^{-\infty} 
\]
%and 
and 
\[
 WF'( \tilde{G}^\pm) 
 %\Lambda_{\mathcal{X}}  \circ  WF'(G^\pm) \circ \Lambda_{\mathcal{Y}} 
  \subset \Lambda_0 \cup \Lambda_1^\pm  \cup  \Lambda_2^\pm   , \quad \Lambda_0 := \{ ( \lambda , \lambda) : \lambda \in \dot{T}^* \R^4\}, \quad \Lambda_1 :=  \Lambda_{D_0}  \cup  \Lambda_{D_0-D_1},\quad \Lambda_2 := \Lambda_{D_1}^{-} \circ \Lambda_1 
\]  
for the conjugated operators $\tilde{G}^\pm :=  \mathcal{X} G^\pm \mathcal{Y}$. 
% microlocally near $\lambda$. 
Moreover, the operators $G^\pm$ are  microlocally unique modulo $I^{-\infty}$. 
%Moreover, these operators are microlocally unique modulo $I^{-\infty}(M)$.  
%Here $\mathcal{E}_\Gamma' (M) := \{ u \in \mathcal{E}'(M) : WF(u) \subset \Gamma \} $. 
%\[
%\square_g E \equiv I  \ \text{on} \ M \setminus \mathscr{H}, 
%\]
%Moreover, $\lambda \in WF' (E) \setminus  \Lambda_0 $ 
%+ADD Propagation of singularities. 
%+ADD orders
\end{corollary}

\subsection{Previous Works} 

 {\color{black}  Operators with multiple characteristics have been studied by several authors, including H\"ormander \cite{Hormadouble}, Uhlmann \cite{Uhlmann-Thesis}, \cite{Uhlmann-Multiple}, \cite{Uhlmann-Light}, Melrose and Uhlmann \cite{Conical-Refraction}, Ivrii \cite{Ivrii}, Boutet de Monvel \cite{Monvel}, \cite{BoutetdeMonvel1974}, Sj\"ostrand \cite{Sjostrand1976},  Petrini and Sordoni \cite{ojm/1200783427}, Lascar \cite{Lascar1981}, and Parenti \& Parmeggiani \cite{doi:10.1080/03605300902892360}. For operators with radial points, see the articles  \cite{radial},  \cite{radial1} by Haber, Vasy and Melrose. 
 
%{\color{black}Microlocal analysis for asymptotically hyperbolic and Kerr-de Sitter spaces has been developed in  \cite{Vasy2010MicrolocalAO}.  
Our work resembles the microlocal and semiclassical studies involving normally hyperbolic trapping. 
See the articles \cite{2011AnHP...12.1349W}, \cite{articledy}, \cite{Dyatlov2013ResonancePA} on semiclassical estimates developed by Wunsch, Zworski and Dyatlov,  \cite{Vasy2010MicrolocalAO} on
microlocal analysis on Kerr-de Sitter spaces by Vasy, and the microlocal propagation theorems in \cite{hintz2021normally}, \cite{jia2022propagation}, \cite{jia2023applications} by Hintz and Jia. 
%for sub-extremal Kerr black holes and Kerr-de Sitter spaces. 
These theories are applicable when angular velocity of the black hole is sub-extremal ($a < r_s/2$). 
 As the two horizons merge in the extremal setting ($a = r_s/2$), the type of blowup in the metric changes. 
Due to this phenomenon, the principal symbol develops critical points at which the canonical symplectic flow in the characteristic variety ``freezes''.
Near these points, the conditions of normally hyperbolic trapping are no longer met.} 

The theory for operators of real principal type has been explored in \cite{Hormander-FIO1, Hormander-FIO2},  \cite{Melrose-Uhlmann}. 
For textbooks on microlocal analysis, see e.g. \cite{Grigis-Sjostrand}, \cite{Taylor-Book}, \cite{Duistermaat-Book}. 
The Kerr metric as an exact solution to the Einstein field equations in vacuum was originally discovered by Roy P. Kerr in \cite{PhysRevLett.11.237} in 1963. 
 {\color{black}Stability of Kerr–de Sitter family of black holes was shown over a half a century later by Hintz and Vasy in \cite{10.4310/ACTA.2018.v220.n1.a1}.  }
For an introduction to the Kerr spacetime, see e.g. \cite{o2014geometry}. 
Other important vacuum solutions are Schwarzchild metric \cite{schwarzschild1999gravitational} (uncharged, non-rotating), Reissner–Nordström metric \cite{https://doi.org/10.1002/andp.19163550905, https://doi.org/10.1002/andp.19173591804} (charged, non-rotating), and the Kerr-Newman \cite{10.1063/1.1704350, 10.1063/1.1704351} (charged, rotating). 
%For an introduction to the Kerr spacetime, see e.g. \cite{o2014geometry}.  
%See also \cite{o2014geometry} for an introduction to Kerr spacetime. 

\subsection{Outline of the Article} 

In Section \ref{preli}, we introduce the notation and some of the basic concepts in microlocal analysis.
In Section \ref{lemmas}, we state and prove the lemmas required in 
Section \ref{mainsection} to carry out the microlocal reduction of the Cauchy problem which is then applied to derive the claim of Theorem \ref{main1}. 

%\subsection{Acknowledgements}

%{\color{black}The author is grateful for the valuable suggestions and feedback from the reviewers.}
%A special thanks to Professor Gunther Uhlmann for interesting discussions regarding the topic. 

\section{Preliminaries}\label{preli}

\subsection{Notation} 

Let $X$ be a smooth manifold. We denote by $TX$ and $T^*X$ the tangent and cotangent bundles of $X$, respectively. 
The fibers of these bundles at a given base point $q\in X$ are denoted by $T_qX$ and $T^*_qX$, respectively. 
We also denote by $\dot{T}X$ and $\dot{T}^*X$ the bundles of the punctured fibers 
$\dot{T}_q X:= T_q X \setminus \{0\}$, $q\in X$,  and $\dot{T}^*_q X:= T^*_q X \setminus \{0\}$, $q\in X$, respectively. 
Analogously, we write $\dot\R^n := \R^n \setminus \{0\}$.
Provided local coordinates $q^1,\dots,q^n$ of $X$,  the derivative $D_\alpha$, $\alpha \in \{ 1,\dots,n\}$ stands for  $-i\partial_\alpha$. Here $i$ is the imaginary unit. 
%
%As usual, the Einstein summation convention is considered. 

\subsection{Lagrangian and Involutive Submanifolds}

Let $X$ be a smooth manifold of dimension $n$ and denote by $\sigma$ the canonical symplectic 2-form $\sigma =d\mo_\mu \wedge  dq^\mu  $ on  $T^*X$. 
A submanifold $\Lambda \subset T^*X$ is called involutive if the symplectic normal 
\[
(T_\lambda \Lambda)^\sigma := \{ v \in T_\lambda T^*X :  \sigma(v,w) = 0, \ \forall w \in  T_\lambda \Lambda \} 
\]
 lies in  $ T_\lambda \Lambda$ for every $\lambda \in \Lambda$. If  $(T_\lambda \Lambda)^\sigma = T_\lambda \Lambda$, $\forall \lambda \in \Lambda$, then $\Lambda$ is said to be a Lagrangian submanifold. Indeed, the submanifold $\Lambda \subset T^*X$ is Lagrangian if and only if it is 
involutive and $\dim (\Lambda) = n$. An important example of a Lagrangian manifold is the conormal bundle $N^* V$ of a submanifold $V\subset X$. 
A submanifold $\Lambda \subset T^*X$ is involutive if and only if $\{ f_j, f_k\} |_{f_1 = \dots = f_k = 0}= 0$ 
for any set $f_1,\dots,f_k$ of local smooth functions with linearly independent $df_j$, $j=1,\dots,k$ defining $\Lambda$ locally as a level set 
$\{ f_1 = f_2= \cdots = f_k = 0 \}$. For a Lagrangian manifold, $k= n$. 

\begin{remark}
One may replace $T^*X$ above with any symplectic manifold. 
\end{remark}

{\color{black}
\begin{definition}
%Let $\Lambda \subset \dot{T}^* X$ be a conic Lagrangian manifold. 
Let $U \subset X$ be an open neighbourhood. A smooth $\varphi  : U  \times \R^k\setminus \{0\} \to \R$ satisfying
\[
d\varphi(q,\xi ) \neq 0, 
\quad \text{and} \quad 
\varphi (q,\alpha \xi) = \alpha \varphi (q,\xi) ,  \quad \text{for all} \quad  (q,\xi) \in U  \times \R^k\setminus \{0\}, \quad \text{and} \quad \alpha \in \R \setminus \{0\},
\]
is called a local  phase function.  In this article, all phase functions are assumed to be homogeneous of degree 1 and real valued. 
Every local phase function generates a conic Lagrangian submanifold $\Lambda_\varphi \subset T^* U \setminus \{0\}$, defined by
\[
\Lambda_\varphi := \{  (q, d_\xi \varphi (q,\xi) )  :  (q,\xi) \in \ker (d_\xi \varphi )   \}.
\]
Given a conic Lagrangian manifold   $\Lambda \subset T^*X$  and $\lambda \in \Lambda$, we say that a local phase function $\varphi  : U  \times \R^k\setminus \{0\} \to \R$ defines $\Lambda$ 
near $\lambda$ if $\Lambda_\varphi = \Lambda \cap \Gamma$ holds in some conic neighbourhood $\Gamma \subset T^*U $ containing $\lambda$. We call $\varphi$ a local phase function defining $\Lambda$. 
\end{definition}
}
%\begin{remark}
% In this article, all phase functions are assumed to be homogeneous of degree 1 and real valued. 
%\end{remark}
%Let $\Lambda \subset \dot{T}^* X$ be a conic Lagrangian manifold. 
%For every $\lambda \in \Lambda$ one can find an open neighbourhood $U \subset X$ of  the base point $\pi(\lambda)$ and a smooth  $\varphi : U \times \R^k \to \R $ satisfying
%\[
%d\varphi(q,\mo) \neq 0, 
%\quad \text{and} \quad 
%\varphi (q,\lambda \mo) = \lambda \varphi (q,\mo) , 
%\]
%such that $\Lambda $ coincides with 
%$\Lambda_\varphi := \{ (q,d_q \varphi(q,\mo) ) : d_\mo \varphi(q,\mo) = 0\}$ in a conic neighbourhood of $\lambda$. 
%We call $\varphi$ a local phase function defining $\Lambda$. 
For $ \Lambda = N^* V$ , one may write locally $N^*V = \{ (q',q'',\mo',\mo'') : q' = 0, \ \mo''=0 \}$ which is defined by $\varphi(q,\mo) = q' \cdot \mo'$. 
%
%Elements in the H\"ormander class $I^m (\Lambda) \subset \mathcal{D}'(X) $ of 

\subsection{Symbol Spaces}

Let $U \subset X$ be open and set $1\leq  k \leq n$. The space $S^{\mu}  (U \times \R^k)$ of symbols  (of type $1,0$) of order $\mu \in \R \cup \{- \infty\}$ is defined as the collection of all $a \in C^\infty(U \times \R^k)$ such that
for all compact $K \subset U$ and multi-indices $\alpha \in \mathbb{N}^n$ and $\beta \in \mathbb{N}^k$ there is a constant $c = c_{K,\alpha,\beta,a} >0$ such that 
\[
| \partial_q^\alpha \partial_\mo^\beta a(q,\mo) | \leq c  (1+|\mo|)^{\mu - |\beta|} , \quad (q,\mo) \in K \times \R^k.
\]
%We set $S^{-\infty}  (U \times \R^k):= \bigcap_{\mu \in \R } S^\mu (U \times \R^k)$. 
Every symbol $a\in S^{\mu}  (U \times \R^k)$ admits an asymptotic development $a \sim  \sum_{j=0}^\infty  a_j$ such that $a_j \in S^{\mu-j}  (U \times \R^k)$. 
Conversely, every such a sum has a representative $a\in S^{\mu}  (U \times \R^k)$ which is unique modulo $S^{-\infty}  (U \times \R^k)$.  
The space  $S^{\mu}_{cl}  (U \times \R^k) $ of classical symbols is defined as those elements of $ S^{\mu}  (U \times \R^k)$ that admit asymptotic development $a \sim  \sum_{j=0}^\infty  a_j $ into symbols $a_j \in S^{\mu-j}  (U \times \R^k)$ that are positively homogeneous in $\mo$ outside some neighbourhood of origin. 
We refer to \cite[Ch. 1]{Grigis-Sjostrand} for more details on symbol spaces. 

%\begin{remark}
In addition to the standard symbol spaces, there is also the space $S^{\mu_1,\mu_2} (U \times \dot\R^{k_1} \times \R^{k_2})$  of elements $a\in C^\infty(U \times \dot\R^{k_1} \times \R^{k_2})$ 
 such that
for all compact $K \subset U$ and multi-indices $\alpha \in \mathbb{N}^n$ and $\beta_j \in \mathbb{N}^{k_j}$, $j=1,2$ there is a constant $c = c_{K,\alpha,\beta_1,\beta_2,a} >0$ such that 
\[
| \partial_q^\alpha \partial_{\xi'}^{\beta_1}\partial_{\xi''}^{\beta_2}  a(q,\xi',\xi'') | \leq c  (1+|\xi'|+ |\xi''|)^{\mu_1 - |\beta_1|}  (1+|\xi''|)^{\mu_2 - |\beta_2|}, \quad (q,\xi',\xi'') \in K \times \R^{k_1}\times \R^{k_2}.
\]
See %\cite{Melrose-Uhlmann}, 
\cite{Greenleaf-Uhlmann} for more details. 
%\end{remark}

\subsection{The Wave Front Set}
For $u \in \mathcal{D}'(\R^n)$ we define the wave front set $WF(u) \subset \dot{T}^*\R^n$ as {\color{black}the} complement of all the covectors $(q,\mo) \in \dot{T}^*\R^n$ for which 
there exists a smooth compactly supported test function $\phi \in C_c^\infty(\R^n)$, non-zero at $q$, and a conic neighbourhood $\Gamma \subset \dot\R^n = \dot{T}_q^* \R^n$ of $\mo$ such that the Fourier transform $\widehat{\phi u}(\mo)$ on $\Gamma$ decays rapidly, {\color{black} i.e.
\[
| \widehat{\phi u}(\mo) | \leq C_{N} (1+ |\mo| )^{-N}, \quad \mo \in \Gamma, 
\]
for arbitrarily large $N \in \N$. Here $C_N $ is a strictly positive constant that may depend on $N$.}
The definition generalises to manifolds via local coordinates. The wave front set $WF(u)$ describes where $u$ fails to be smooth. 
It captures not only the singular points $q$ (singular support) but also the directions $p$ in which the peaks are visible.  
%The wave front set of a distribution not only captures points where $u$ fails to be smooth but also the direction where it faces at. 
For $u \in C^\infty$ we have $WF(u) = \emptyset$. Such elements are considered residual in the microlocal framework. 
%
%Let $u \in \mathcal{D}'(X)$. 

%We say that $u \in \mathcal{D}'(X)$ has a given (reasonable) property microlocally near $\lambda \in \dot{T}^* X$ if there is $v \in \mathcal{D}'(X)$ with this property and a conic neighbourhood $\Gamma $ of $\lambda $ in $\dot{T}^* X$ such that
%\begin{equation}\label{dsafff}
%WF(u-v) \cap \Gamma = \emptyset. 
%\end{equation}
%Similarly, we say that something holds microlocally away from $W \subset \dot{T}^* X$ if for every $\lambda \in \dot{T}^* X \setminus W$ there is a conic neighbourhood $\Gamma $ of $\lambda $ in $\dot{T}^* X \setminus W$ such that \eqref{dsafff} holds for some $v$ with this property.  

\subsection{Lagrangian Distributions}

A Lagrangian distribution of order $m$ over a conic Lagrangian submanifold $\Lambda \subset \dot{T}^*X$  is defined as an distribution $u \in  \mathcal{D}'(X)$ 
that can be expressed as a locally finite sum  $u= \sum_\alpha u_\alpha$ of oscillatory  integrals of the form 
\begin{equation}\label{oscu}
u_\alpha(q) = \int_{\mathbb{R}^k} e^{i\varphi(q,\zeta) } a  (q,\zeta) d\zeta_1,\dots d\zeta_k  
\end{equation}
where $\varphi = \varphi_\alpha$ is a smooth homogeneous phase function defining $\Lambda$  and $a=a_\alpha \in S^{m-\frac{k}{2} + \frac{n}{4}}( X \times \dot\R^k)$. 
It follows that $u$ is smooth microlocally away from $\Lambda$. Indeed, $WF(u)\subset \Lambda$ for the wave front set of $u$ and it is the order of the symbol $a$ in the oscillatory representation above that describes how strong the singularity is microlocally near given point of $\Lambda$. 
The order $-\infty$ corresponds to an absence of singularity.  
We denote by $I^m (\Lambda)$ the space of Lagrangian distributions over $\Lambda$. Similarly, we denote by $I^m_{cl} (\Lambda)$ the space of Lagrangian distributions over $\Lambda$ with classical symbols. 
%In fact, we shall use only the classical spaces in this article. 
An element of $I^m(N^*V)$, where $V$ is a submanifold of $X$ is referred to as a conormal distribution, or more precisely, a distribution conormal to $V$.

\subsection{Operator Classes}

Fourier integral operator (FIO) is by definition an operator that admits a Lagrangian distribution as the Schwartz kernel. In this article, it is also required that the operator is properly supported. 
A Fourier integral operator associated with the Lagrangian $\Lambda= N^*\{ (q,\tilde{q}) \in X \times X  : q= \tilde{q} \} =  \{ (q,\mo ; \tilde{q},\tilde\mo) \in \dot{T}^*X \times \dot{T}^*X : q= \tilde{q} , \ \tilde\mo = - \mo \}$ is called a pseudo-differential operator ($\Psi$DO). The space of (properly supported) pseudo-differential operators on $X$ of degree $m$ is denoted by $\Psi^m(X)$. 

For a FIO, the leading term on $\Lambda$ in the asymptotic sum $a \sim \sum_{\alpha=0}^\infty a_\alpha$ of the underlying symbol is referred to as the (local) principal symbol of the operator. 
Globally it is identified as a section in the Maslov line bundle $\Omega_{1/2 } \otimes \mathcal{L}$ (See e.g. \cite[Ch. 11]{Grigis-Sjostrand}, \cite[Ch. 4]{Duistermaat-Book}). The principal symbol fixes an unique equivalence class in $I^m (\Lambda) / I^{m-1} (\Lambda)$ and vice versa. An elliptic Fourier integral operator is an operator for which the principal symbol is non-vanishing on the whole canonical relation. This is said to hold microlocally near a point if the condition is satisfied in a conic neighbourhood of it. 

For two FIOs $F : C_c^\infty(Y) \to \mathcal{D}' (X)$ and $G : C_c^\infty(Z) \to \mathcal{D}' (Y)$ with kernels $k_F \in I^{m_F}(\Lambda_F)$ and $k_G \in I^{m_G}(\Lambda_G)$, we have for the composition $F \circ G$ (if well defined) the rule 
\begin{equation}\label{compoeq}
{\color{black}k_{F \circ G} \in I^{m_F+ m_G {\color{black}+ e/2}} ( \Lambda_{ F \circ G }),}
\end{equation}
{\color{black}where $\Lambda_{F \circ G}' =  \Lambda_F' \circ \Lambda_G' := \{ (\lambda_1,\lambda_3)  :  \exists \lambda_2 ; \ (\lambda_1,\lambda_2) \in \Lambda_F' , \ ( \lambda_2, \lambda_3) \in \Lambda_G' \}$ }
and the prime refers to the associated canonical relation\footnote{Lagrangian manifold with respect to the twisted symplectic form} $\Lambda' := \{(\lambda_1,\lambda_2) :  (\lambda_1,-\lambda_2) \in \Lambda \}$. 
{\color{black}Above $e$ stands for the excess in the intersection  $(\Lambda_F' \times  \Lambda_G')  \cap (T^*X \times \Delta \times T^* Z)$ which is required to be clean. Here $\Delta$ stands for the diagonal $\{ (\lambda,\tilde{\lambda})  : \lambda = \tilde{\lambda} \}$ in $T^*Y \times T^*Y$. For a transversal intersection, we have $e=0$.   }
Sufficient conditions for the compositions to be well defined and the associated symbol calculus (transversal intersection calculus) is explored in \cite{Hormander-FIO1} \cite{Duistermaat-Book}, \cite{Grigis-Sjostrand}. 
The formula above applies also to the case where $F$ is a Lagrangian distribution. This is obtained by setting $Z$ to be a single point space. 
Ellipticity is preserved in compositions and each elliptic operator admits an elliptic inverse modulo a residual with  smooth kernel.  Moreover, an elliptic FIO preserves  the wave front set  in the sense that $WF(F u) = \Lambda_F' \circ WF(u)$, provided that the composition is well defined.  Thus, ellipticity can be considered as an isomorphic relation in microlocal analysis (cf. diffeomorphisms in the category of smooth manifolds or bijections in topology). 
It is natural to associate the operators with the canonical relations instead of the Lagrangian manifolds and we adopt the same notation 
$ I^m (\Lambda)$ for operators if $\Lambda$ is a canonical relation, meaning that the Schwartz kernel is a Lagrangian distribution in $I^m (\Lambda')$.

{\color{black}
%\begin{remark}
Associated to a pair $\Lambda_0, \Lambda_1 \subset \dot{T}^*X$ of cleanly intersecting pair of Lagrangian manifolds, there is the space $I^{m_1,m_2} ( \Lambda_0, \Lambda_1) \subset \mathcal{D}'(X)$, first introduced by Melrose and Uhlmann in \cite{Melrose-Uhlmann}. We follow the formalism in \cite[Section 1]{Greenleaf-Uhlmann}. 
%(see also [G-U ESTIMATE PAPER]).  
Let us consider the case where $\Lambda_0, \Lambda_1$ are locally of the form $N^*S_2$ and $N^*S_1$, where $S_2 \subset S_1 \subset X$ are nested manifolds of codimension $k_1$ and $k_1 + k_2$. 
We can choose local coordinates $z = (z',z'',z''') \in \R^{k_1} \times \R^{k_2} \times  \R^{n-k_1-k_2}$ such that $S_1 = \{ z' = 0 \}$ and $S_2 = \{ z' = 0, z'' = 0 \}$. 
The elements of $I^{m_1,m_2} ( \Lambda_0, \Lambda_1)$ can be microlocally reduced into a finite sum of oscillatory integrals 
\[
\int_{\R^{k_1}} \int_{\R^{k_2}} e^{z' \zeta' + z'' \cdot \zeta' } a(z',z'',\zeta',\zeta'') d\zeta' d\zeta'', \quad a \in S^{\mu_1,\mu_2} ( \R^n \times \dot\R^{k_1} \times \R^{k_2} )
\]
where $\mu_1 = m_1-k_1/2 + n/4$ and $\mu_2 = m_2 - k_2 /2$. 
%by conjugating with FIOs corresponding (see Egorov's theorem \cite[Theorem 10.1]{Grigis-Sjostrand}) to a canonical transformation  that takes $(\Lambda_0,\Lambda_1)$ locally into intersection 
%
%($T_0^* \R^n, \{ q_1 \geq 0, \ q' = 0, \ \xi_1 = 0 \})$.   
Microlocally away from $   \Lambda_1 $ (resp. $ \Lambda_0 $) the elements in $I^{m_1,m_2} ( \Lambda_0, \Lambda_1)$ lie in $I^{m_1+ m_2}( \Lambda_0) $ (resp. $I^{m_1 }( \Lambda_1) $).  }
These distributions define operator classes together with a transversal composition calculus similar to the one developed for FIOs. 

%Again, we adopt the operator class notation $I^{m_1,m_2}(\Lambda_0, \Lambda_1)$ if $\Lambda_0, \Lambda_1$ are cleanly intersecting canonical relations. 
 A (microlocal) inverse of a smooth Lorentzian wave operator $\square_g$ has a kernel in the operator class $I^{{\color{black}-3/2,-1/2}}(\Delta, \Lambda_\square )$, where $\Delta$ stands for the diagonal $\Delta= \{ (\lambda,\tilde{\lambda})  : \lambda = \tilde{\lambda} \}$ and 
 $\Lambda_\square
 %= \{ ( \gamma_{q,\mo^\sharp } (s) ,  -\dot\gamma_{q,-\mo^\sharp }^\flat (s) ; q,\mo) : g(p,p) = 0, \ s \in \R \}  
 $ is the flowout canonical relation which consists of the pairs that are connected to each other via the geodesic flow (see \eqref{flowa} for $\mathcal{H} = -\frac{1}{2} g^{\mu \nu} \xi_\mu \xi_\nu$ ) along the lightcones.\footnote{Notice that $\Delta $ and $ \Lambda_\square$ intersect cleanly. That is; $\Delta \cap  \Lambda_\square$ is a smooth submanifold and $T_\lambda(  \Delta \cap  \Lambda_\square) = T_\lambda \Delta \cap T_\lambda  \Lambda_\square$ holds at every intersection point $\lambda \in \Delta \cap  \Lambda_\square$.}
  The result extends (see \cite{Hormander-FIO1},\cite{Melrose-Uhlmann}) to real principal type operators. 
 %, where $\gamma_{q,v}$ stands for the geodesic emanating from $  (q,v) = (  \gamma_{q,v}(0), \dot\gamma_{q,v}(0))$. 
 The inverse, also known as parametrix, has its operator wave front set\footnote{The operator wave front set stands for the twisted $WF'(k):= (WF(k))'$ where $k$ is the kernel of the operator. } in $\Delta \cup \Lambda_\square$.  
%\end{remark}

\section{Lemmas}\label{lemmas}
%Identifying $C^\infty $ with the half-densities 
A covector $\lambda$ is called a double characteristic point of a pseudo-differential operator $R$ if it satisfies both $r(\lambda) = 0 $ and $d r(\lambda) = 0$ where $r$ stands for the principal symbol of $R$ near $\lambda$. 
\begin{lemma}\label{proporo}
%Let $P_0(q,\mo) =-   \sqrt{ -\det g} \Delta g^{\mu \nu} \mo_\mu \mo_\nu  \in S_{cl}^2 $ be the positively homogeneous principal symbol of $P \equiv \square_{\Delta q} $ in the local form $($i.e. $\sigma_P  =  P_0  \sqrt{ \wedge_\mu  d\mo_\mu })$. 
%%\mod \Psi_{cl}^1(\dot{M})$. 
%%In a conic neighbourhood  $\Gamma \subset \dot{T}^*\dot{M}$ with  $\Gamma \cap N^* \mathscr{H} = \emptyset$ 
%Let $\sigma_P$ be the classical principal symbol of $P \equiv \square_{\Delta g}$. 
%Let $P_0$ be the classical principal symbol $P_0 =  - \sqrt{-\det g} \Delta g^{\mu \nu } \mo_\mu \mo_{\nu} $ of $P$ in local canonical coordinates.  
The set $\Sigma_2$ equals the double characteristic points of $P$. 
%\[
%\{ P_0 = 0, \ d P_0 = 0\}.
%\]
%in that neighbourhood. 
%of $P \equiv \square_{\Delta g}$. 
%on $\dot{T}^* \dot{M} \setminus N^* \mathscr{H}$. 
\end{lemma}
\begin{proof}
Let us first study the claim in a neighbourhood that does not meet the covectors over the axis $\mathscr{A}$. 
In Boyer-Lindquist coordinates, $P \equiv P_0(q,D) \mod \Psi^1(\R^4) $, where 
%\[
% P(q,\tilde{q})    \equiv   \int_{\R^4} e^{i (q-\tilde{q})   \mo } P_0(q,\mo)  d\mo ,
%\]
%modulo first order terms, 
%microlocally in Boyer-Lindquist coordinates $P  \equiv  P_0(q, D_q)    \mod \Psi^1_{cl}$ for 
\begin{equation}\label{adorwe}
P_0 (q,\mo) = - \sqrt{-\det g} \Delta g^{\mu \nu } \mo_\mu \mo_{\nu}  =   -(  \sqrt{-\det g} \  \Delta g^{tt})  \left(  ( \mo_t + \Psi )^2 -  \Delta \Phi \right) , 
\end{equation}
\begin{equation}\label{psieq}
\Psi =  \frac{g^{t \phi} }{g^{tt}}  \mo_\phi =  \frac{ac  }{  ( r^2 + a^2 +   \frac{r_s r   }{\Sigma}  a^2  \sin^2 \theta ) }    \frac{r_s r   }{\Sigma}  \mo_\phi  ,
%  \in S^2_{cl} ( U \times \dot\R^{4} )
\end{equation}
and
\begin{equation}\label{phieq}
\Phi = \frac{1}{\Delta} \left( \Psi^2 -   \frac{1}{g^{tt}}( g^{rr} \mo_r^2 + g^{\theta \theta } \mo_\theta^2 + g^{\phi \phi} \mo_\phi^2 ) \right) =   \frac{ c^2 }{  (r^2 + a^2 + \frac{r_s r   }{\Sigma}  a^2 \sin^2 \theta )  } \left( \frac{  \Delta }{\Sigma} \mo_r^2 + \frac{  1 }{\Sigma} \mo_\theta^2  + \frac{1 }{\Sigma \sin^2 \theta}   \mo_\phi^2 \right) .
%\in S^2_{cl}( U \times \dot\R^{4} ).
\end{equation}
Recall that $g^{tt}$ is strictly negative with a type $O(\frac{1}{\Delta })$ blowup and the local volume element $\sqrt{-\det g} \  dq$ does not vanish outside the ring singularity $\{ \theta = \pi/2, \ r= 0\}$ which we have excluded from $M$. 
Indeed, the coefficient $ -(  \sqrt{-\det g} \  \Delta g^{tt}) $ is smooth and non-vanishing on the coordinate chart so it suffices to prove the claim for $   ( \mo_t + \Psi )^2 -   \Delta \Phi $ instead of $P_0$.
We also recall that $\square_g$ and hence also $P$ is of real principal type away from the horizon so we may focus on points with $r=r_s/2$. 
Characteristic condition $ ( \mo_t + \Psi )^2 -   \Delta \Phi =0$ at such points is equivalent to $\mo_t + \Psi = 0$. Moreover,
\[
 d\left(  ( \mo_t + \Psi )^2 -   \Delta \Phi  \right)  |_{r=r_s/2}=( 2 ( \mo_t + \Psi ) d( \mo_t + \Psi )  - 2 (r-r_s/2)  \Phi dr   - \Delta d\Phi ) |_{r=r_s/2}  = 2 ( \mo_t + \Psi ) d( \mo_t + \Psi ) |_{r=r_s/2} .
\]
As $\partial_{\mo_t}( \mo_t + \Psi )=1 \neq 0 $, this derivative vanishes if and only if $  \mo_t + \Psi  = 0$. In conclusion, the double characteristics in Boyer-Lindquist coordinates coincide with $\Sigma_2$. 
Let us now deduce the claim near $ \dot{T}^*_{\mathscr{A} } M:= \{ (q,\mo) \in \dot{T}^* M : q \in \mathscr{A} \}$. First of all, the double characteristic set as a global entity is topologically closed in $\dot{T}^*M$ as a preimage of the closed set $\{0\}$ in a continuous map. 
Combining this with the first part of the proof implies that the topological closure of $\Sigma_2 \setminus \dot{T}^*_{\mathscr{A} } M$ in $\dot{T}^*M$ lies in the double characteristic set. 
One checks that this closure is the whole $\Sigma_2$. Conversely, the set $\Sigma_2$ is closed in $\dot{T}^*M$ too so a similar argument implies that the double characteristics lie in $\Sigma_2$. 
This finishes the proof. 

%On the other hand, the characteristic covectors on the axis that do not intersect $\Sigma_2$ lie outside the horizon. 
%Again, the operator is of real principal type near such points. 
\end{proof}

\begin{lemma}\label{involemma}
The set $\Sigma_2  \subset \dot{T}^* M$ is an involutive submanifold of codimension 2. 
\begin{proof}
Recall \eqref{adorwe} from Lemma \ref{proporo}. 
By definition, $\Sigma_2$ is the level set of the smooth functions $f_1 = r-r_s/2$ and $f_2 = \mo_t + \Psi$ on $\dot{T}^* M$, both extending smoothly onto $ \dot{T}^*_{\mathscr{A} } M$. 
The derivatives $df_1 = dr$ and $df_2 = d\mo_t + d_{r,\theta,\mo_\phi }\Psi$ are clearly non-vanishing and linearly independent. 
Moreover, the Poisson bracket $\{ f_1 , f_2\} $ is identically zero. Hence, the claim is a consequence of \cite[Lemma 3.6.1]{Duistermaat-Book}. 

\end{proof}
\end{lemma}

\begin{lemma}\label{hesslemma}
The Hessian of the principal symbol of $P$ has rank $2$ on $\Sigma_2 \setminus N^* \mathscr{H}$ 
\end{lemma}
\begin{proof}
Again, we recall the notation from Lemma \ref{proporo}. 
%The manifold $\Sigma_2$ meets  $\dot{T}_{\mathscr{A}}^* (M) :=  \{  (q,\mo) \in \dot{T}^* (M) : q \in \mathscr{A} \}$ in the excluded $\{ \mo_t = 0\}$ so it suffices to 
%We work in Boyer-Lindquist coordinates. Similar arguments apply in local coordinates near the poles. 
After dividing with a non-vanishing function we may assume that $P \equiv \tilde{P}_0(q,D_q)$ for $\tilde{P}_0 = (\mo_t + \Psi)^2 - \Delta \Psi$. 
One computes
\[
Hess (P_0) |_{\Sigma_2}  = 2 d\mo_t \otimes \partial_{\mo_t} - 2\Phi|_{\Sigma_2} dr \otimes \partial_r.
\]
Away from $ N^* \mathscr{H}= \{ r=r_s/2 , \ \mo_t =  \mo_\theta  = \mo_\phi = 0\} $, the points in $\Sigma_2$ must satisfy $\Phi  \neq 0$. 
In fact, $N^* \mathscr{H} = \text{ker} \Phi$. One checks that this extends to the poles too. 
Hence $\mo_\phi \neq 0$ which implies that the latter term above does not vanish. 
Thus, the claim holds. 
\end{proof}

Below we denote $\dot{T}_{\mathscr{H}}^* (M) := \{  (q,\mo) \in \dot{T}^* (M) : q \in \mathscr{H} \} \supset \Sigma_2 $. 
\begin{lemma}\label{sublemma}
The subprincipal symbol $c_P$ $($see e.g. \cite[Theorem 11.11]{Grigis-Sjostrand}$)$ of $P $ vanishes on $\dot{T}_{\mathscr{H}}^* (M) $. In particular,  $c_P |_{\Sigma_2} = 0$. 
\begin{proof}
It suffices to show this in Boyer-Lindquist coordinates. The general claim then follows by continuity. 
By definition, the subprincipal symbol $c_P$  of 
\[
P :=  \Delta \square_g+ \kappa = \sqrt{- \det g} \Delta g^{\mu \nu } \partial_\mu \partial_\nu +    \Delta \partial_\mu (  \sqrt{-\det g} g^{\mu \nu } )  \partial_\nu  + \kappa
\]
(in coordinates) is
\begin{align}
c_P =  -i \Delta  \frac{\partial}{\partial q^\mu} ( \sqrt{-\det g} g^{\mu \nu} \mo_\nu )  - i   \frac{\partial}{\partial q^\mu}  ( \Delta \sqrt{-\det g} g^{\mu \nu }  \mo_\nu ) . 
 \end{align}
% Since $\Delta $ depends only on $r$ and $g^{r\mu} = \delta^{r\mu} \frac{\Delta}{\Sigma} $, this reduces to
Since the metric stationary and rotationally symmetric, the derivatives above vanish except when $\mu$ corresponds to $r$ or $\theta$. 
Moreover, $g^{r \nu } =  \delta^{r\nu} \frac{\Delta}{\Sigma} $, $g^{\theta \nu } =  \delta^{\theta \nu} \frac{1}{\Sigma} $ and $\sqrt{-\det g} = \Sigma \sin \theta$.  
Thus, 
 \begin{align}
c_P =  -i 6  \sin \theta \  (r-r_s/2)^3 \mo_r   -2 i \Delta  \cos \theta  \mo_\theta   
%  -2 i \sum_{\mu \neq r,\theta}   \sum_{\nu \neq r,\theta}  \frac{\partial}{\partial q^\mu} (  \Delta \sqrt{-\det g} g^{\mu \nu} \mo_\nu )  
\end{align}
which vanishes for base points on the horizon $\{ r = r_s/2\} = \{ \Delta = 0\}$. 
% \begin{align}
%-2 i \sum_{\mu \neq r,\theta}   \sum_{\nu \neq r,\theta}  \frac{\partial}{\partial q^\mu} (  \Delta \sqrt{-\det g} g^{\mu \nu} \mo_\nu )  = -2i  \frac{\partial}{\partial q^\mu} (  \Delta \sqrt{-\det g} g^{\mu \nu} \mo_\nu ) 
%\end{align}

\end{proof}
\end{lemma}

\section{Proof of Theorem \ref{main1}}\label{mainsection}

We can now construct the parametrix by using the lemmas above. The proof is based on \cite{Uhlmann-Thesis} (see also \cite{Conical-Refraction}). 
%We describe the construction and explain how the additional canonical relation $\Lambda_2$ is defined in this framework. 
The explicit formula \eqref{cano2} is derived as the last part of the proof. 

By Lemma \ref{proporo}, $\Sigma_2$ coincides with the set of double characteristics of $P$. 
%It suffices to study the case $\lambda \in \Sigma_2 \setminus  N^* \mathscr{H}$ as elsewhere in $\dot{T}^*M \setminus  N^* \mathscr{H}$ the operator $P$ is microlocally elliptic or of real principal type. 
In Boyer-Lindquist coordinates,  $P$ admits the principal symbol
\[
P_0 =  \alpha[  ( \mo_t + \Psi)^2 - \Delta \Phi ], 
\]
where $\alpha = \sqrt{- \det g} \Delta g^{tt}$ is non-vanishing and smooth. 
%Moreover, $\{ \mo_r + \Psi,  \Delta \Phi \} = 0$ 
At  $\lambda
%=  (\hat{q},\hat\mo)
 \in \Sigma_2 \setminus N^* \mathscr{H} $ we get $ \Phi(\lambda) \neq 0$ and hence $\Phi \neq 0$ near such a point by continuity. 
Thus, near any point in $\lambda \in \Sigma_2  \setminus N^* \mathscr{H} $,  we can split $P_0$ into
\[
P_0 =  \alpha  P_0^+ P_0^-
\]
where 
\[
P_0^\pm  :=   \mo_t + \Psi \pm  (r-r_s/2)  \sqrt\Phi 
\]
are smooth and positively homogeneous of degree $1$. The coefficient $\alpha$ will not play any role in our analysis and can be omitted. 
More generally, multiplying with any elliptic pseudo differential operator leaves the canonical relation and (operator) wave front set intact. 
It follows from Lemmas \ref{proporo}-\ref{hesslemma} that $\{ P_0^+, P_0^-\}|_{\Sigma_2} = 0$ and that the Hamiltonian vector fields $H_{P_0^+}$, $H_{P_0^-}$ and the cone axis are linearly independent in a conic neighbourhood of $\lambda$.  
Moreover, $\Sigma_1= \{ P_0^+ = 0 \ \text{or} \ P_0^- = 0\}$ and $\Sigma_2 = \{ P_0^+ = 0 =  P_0^-\}$ there. Recall also from Lemma \ref{sublemma} that the subprincipal symbol of $P$ vanishes on $\Sigma_2$ which is required for the Cauchy problem to be uniquely solvable. 
Applying \cite[Proposition 2.2.]{Uhlmann-Thesis}, we transform $P$  
microlocally near given $\lambda$ to the model form 
\[
\tilde{P} = \mathcal{F} P  \mathcal{G} =   D_1 D_2 + A_1 D_1 + A_2 D_2 + R, 
   \quad A_1,A_2, R \in \Psi^0_{cl}(\R^4) , \quad D_\mu := -i\partial_\mu
\]
% $D_\mu := -i\partial_\mu$, 
%$A_1,A_2, R \in \Psi^0_{cl}(\R^4)$ 
by conjugating with elliptic Fourier integral operators of order $0$, indicated here by $\mathcal{F},\mathcal{G}$. 
The new local coordinates\footnote{We may assume that $\lambda $ is taken to $ \dot{T}_0^*\R^4$.}  $(z^1,\dots,z^4$, $\eta_1,\dots,\eta_4 )$ in $\dot{T}^* \R^4$ 
%(we may assume that $\lambda $ is taken to $ \dot{T}_0^*\R^4$)   
are related to the original ones\footnote{E.g. the canonical pairs $q^\mu,p_\nu$ generated by Boyer-Lindquist coordinates $q=(t,r,\theta,\phi)$.} via a local homogeneous canonical transformation.
Moreover, the derivatives $D_1$ and $D_2$ are canonical transformations of $f_+ P_0^+(q,D) $ and $f^- P_-(q,D) $, respectively, in the sense of Egorov's theorem (see \cite[Theorem 10.1.]{Grigis-Sjostrand}) and 
$f^\pm \in S^0_{cl}$ are some homogeneous functions such that 
\[
\{ f^+ P_0^+ , {\color{black}f^- P_0^-} \} = 0, \quad  \text{and} \quad  f^\pm \neq 0
\]
hold in a conic neighbourhood of $\lambda$. Above $\{ \cdot, \cdot \}$ stands for the Poisson bracket in $\dot{T}^* M$. 
 %coincide with $f^+P_0^+  \circ \tilde\chi $ and $f^- P \circ \tilde\chi$ 
A parametrix for the Cauchy problem related to $\tilde{P}$ is an operator $\tilde{E} : \mathcal{E}'(\R^3) \times \mathcal{E}'(\R^3) \to \mathcal{D}' ( \R^4) $ satisfying $\tilde{E} (\tilde{f}_1,\tilde{f}_2) \equiv \tilde{u} \mod C^\infty$, where we identify $  \R^3 = \{ z^1 = 0\} \subset \R^4  $.     This is equivalent, microlocally and modulo $ I^{-\infty}$, to 
% microlocally modulo $I^{-\infty}$,
%\begin{align}
$ \tilde{P}  \tilde{E}   \equiv 0$,  
%\quad 
$ ( \tilde{E}, D_1 \tilde{E})  |_{ z^1 = 0  }  \equiv  (I,I) $, 
%\mod C^{\infty} \\
%\mod I^{-\infty} ( \R^3  ; \R^3 ) 
% \partial_2 u|_{z^2 = 0}   \equiv  f_2 \mod C^{\infty}
%\end{align} 
where 
%we identify $  \R^3 = \{ z^1 = 0\} \subset \R^4  $ and 
 $I$ stands for the identity operator. 
 We may rewrite this as the following first order system
\begin{align}
\mathcal{P}  \mathcal{E}  \equiv 0,  \quad 
%\mod I^{-\infty} \\ 
 \mathcal{E} |_{z^1 = 0} \equiv \mathcal{I},  \label{caussif}
 %\mod   I^{-\infty} 
\end{align}
where $\mathcal{E}$ is a $2\times2$ matrix with operator indices, 
%$ \mathcal{I} :=  \begin{bmatrix} I & 0 \\ 0 & I  \end{bmatrix}$. 
\[
 \mathcal{P} : = 
\begin{bmatrix} 
D_2 + A_1 & A_2 D_2 + R \\
-I & D_1 
\end{bmatrix},
\quad 
 \mathcal{I} :=  \begin{bmatrix} I & 0 \\ 0 & I  \end{bmatrix}.
%
%\mod  \mathcal{M}_{2\times 2} ( \Psi^0 (\R^4) ) 
\]
The two representations are linked via  $\tilde{E}(\tilde{f}_1,\tilde{f}_2)  = \mathcal{E}^{22} \tilde{f}_1 + \mathcal{E}^{21} \tilde{f}_2  $. 
Further, the Cauchy problem \eqref{caussif} admits (see \cite[Section 3]{Uhlmann-Thesis}) the following model form
\begin{align}
\tilde{\mathcal{P}}  \tilde{\mathcal{E} } \equiv 0,  \quad 
%\mod I^{-\infty} \\ 
 \tilde{\mathcal{E}} |_{x^0 = 0} \equiv \mathcal{I},  \label{caussimo}
 %\mod   I^{-\infty} 
\end{align}
%
% the Cauchy problem for  operators with double characteristics is reduced microlocally into 
%$E \equiv E_1+E_2+E_3 \mod I^{-\infty}$ 
%a similar problem for the vector valued operator 
\begin{equation}\label{2194sdddgg}
\tilde{\mathcal{P} } \equiv 
\begin{bmatrix} 
D_0-D_1 & 0 \\
0 & D_0
\end{bmatrix}
\mod  \mathcal{M}_{2\times 2} ( C^\infty( \R ;   \Psi^0 (\R^3)) ) 
\end{equation}
%The transformation applies
which is obtained by applying elliptic pseudodifferential operator matrices of degree $0$.  
The local canonical coordinates 
\begin{equation}\label{Mia_coordinates}
(x,\xi) =  (\overbrace{x^0,x'}^{x}, \overbrace{\xi_0,\xi'}^\xi ) =
(x^0,\overbrace{x^1,x^2,x^3}^{x'},\xi_0,\overbrace{\xi_1,\xi_2,\xi_3}^{\xi'})  \in  U \subset \dot{T}^* \R^4 
\end{equation}
of $\dot{T}^* M$ 
used here are uniquely determined by 
\[
\begin{cases}
x^0 = z^1,\\ x^1 = z^1 - z^2, \\ x^2=z^3,  \\ x^3=z^4. 
\end{cases}
\]
%Thus, we are left to show the following:
%
%\begin{proposition}\label{wedgeprop}
%The  formula \eqref{cano2}  holds. 
%\end{proposition}
%
Indeed, the change of coordinates formula $ \frac{\partial  }{\partial x^\nu } =  \frac{\partial z^\mu }{\partial x^\nu }   \frac{\partial  }{\partial z^\mu }$ implies $\xi_\nu =  \frac{\partial z^\mu }{\partial x^\nu } \eta_\mu$ which gives 
\begin{equation}\label{dosw22}
\begin{cases}
\xi_0 = \eta_1 + \eta_2,  \\
\xi_1 = -\eta_2,\\
\xi_2 = \eta_3,\\
\xi_3 = \eta_4. 
\end{cases}
\end{equation}
Parametrix for the system \eqref{2194sdddgg} is constructed in \cite[Section 4]{Uhlmann-Thesis}. The oscillatory integrals (\ref{cicada1}-\ref{cicada3}) follow from (4.4), (4.5), (4.6) in the reference. 
The phases in \eqref{cicada1} and \eqref{cicada2}  are
\begin{align}
&\varphi_1 = (x'-\tilde{x}')\cdot \xi', \label{varwp} \\ 
&\varphi_2 = x^0 \xi_1 +  (x'-\tilde{x}')\cdot \xi', \label{varwp1}
%&\varphi_3 =  \frac{r+x^0}{2}\xi_1 + i(x'-\tilde{x}')\cdot \xi'  
\end{align}
these generate the canonical relations 
\begin{align}
& \Lambda_{\varphi_1}' = \{ (x,\xi,\tilde{x}', \tilde\xi' ) :      x' = \tilde{x}' , \ \xi_0 = 0 , \  \xi' = \tilde{\xi}' \} \\ 
&\Lambda_{\varphi_2}' = \{ (x,\xi,\tilde{x}', \tilde\xi' ) :     x^1 +x^0 = \tilde{x}^1  , \ (x^2,x^3 ) = (\tilde{x}^2,\tilde{x}^3),  \ \xi_0 = \xi_1 , \  \xi' = \tilde{\xi}' \} 
%&\Lambda_{\varphi_3}' =  \{ (x,\xi, \tilde{x},\tilde{\xi} ) :    x_1 = \tilde{x}_1 + \frac{x_0 - \tilde{x}_0 + r}{2}  ,  \  \xi_0 =  \tilde\xi_0 =  \xi_1 = \tilde\xi_1 = 0 , 
%\\  &  \tilde{x}_0 - x_0   \leq r  \leq x_0-\tilde{x}_0, \  (x_2,x_3, \xi_2,\xi_3)= (\tilde{x}_2,  \tilde{x}_3,\xi_2,\xi_3) \} 
\end{align}
It is straightforward to check that these coincide with 
%the characteristic flowout canonical relations
%\footnote{I.e. pairs that are connected via the flow \eqref{flowa} where $\mathcal{H}$ is the generating function.} 
$C_{D_0}$ and $C_{D_0 - D_1}$ generated by $\xi_0$ and $\xi_0 - \xi_1$, respectively. In conclusion, the terms (\ref{cicada1}) and (\ref{cicada2}) lie in 
$I^{a} ( C_{D_0} )  $ and $I^{a} ( C_{D_0-D_1} )  $, where $a = 0 + {\color{black}3}/2 - 7/4 = {\color{black}-1}/4$. In particular, they lie in $I^{{\color{black} 1/4, -1/2 }} ( C_{D_0}, \Lambda_{D_1} \circ C_{D_0} )  $ and $I^{{\color{black}1/4,-1/2}} ( C_{D_0-D_1}, \Lambda_{D_1} \circ C_{D_0-D_1} )  $, respectively. {\color{black}This follows from \cite[eq. (1.4)]{Greenleaf-Uhlmann}. }
%%
%%
%%In \cite{Uhlmann-Thesis} the microlocal parametrix for a such operator is constructed by solving first the associated Cauchy problem and then applying Duhamel's principle.  
%Recall that $f_+ P_0^+ $ and $f^- P^-_0 $ are related to $\eta_1$ and $\eta_2$  via 
%\begin{align}
%&\eta_1 =  f^+ P_0^+   \circ \chi(z,\eta) \\
%& \eta_2 =  (  f^- P_0^- )\circ \chi(z,\eta) ,
% \end{align}
%   where $\chi$ is a local homogeneous canonical transformation. Hence, by \eqref{dosw22}, 
%%\begin{align}
%%&x_1 = z_1,  &\xi_1 = \zeta_1 + \zeta_2, \\
%%&x_2 = z_1 - z_2, &\xi_2 = \zeta_2 ,\\
%%&x_3 = z_3, & \xi_3 = \zeta_3, \\
%%&x_4 = z_4,  & \xi_4 = \zeta_4,
%%%\xi_1 = \zeta_1 + \zeta_2 \\
%%%\xi_2 = \zeta_2 \\
%%%\xi_3 = \zeta_3 \\
%%%\xi_4 = \zeta_4
%%\end{align}
%\begin{align}
%&\xi_0 =  ( f^+ P_0^+ + f^- P_0^- )  \circ \chi(x,\xi) \\ 
%& \xi_1 =  -( f^- P_0^- )\circ \chi(x,\xi) 
% \end{align}
%It is a straightforward to check that $ \Lambda_{\varphi_1}' $ and $ \Lambda_{\varphi_2}' $ coincide with the flowout canonical relations\footnote{I.e. pairs that are connected via the flow \eqref{flowa} where $\mathcal{H}$ is the generating function.} generated near $\lambda$ by the functions $\xi_0$ and $\xi_0 -\xi_1$, respectively. 
% Moreover, $f^\pm$ are non-vanishing so they do not play any role in the relation. 
% In conclusion, $ \Lambda_{\varphi_1}' $ and $ \Lambda_{\varphi_2}' $ are local expressions of the characteristic flowout canonical relations generated by $P_0^+$ and $P_0^-$. 
% Away from the double characteristics, these relations lie in $C_1$ generated by $P_0$. 
  %
 %
 %
 %
 Let us study the term (\ref{cicada3}). 
 %Below $\Lambda_{D_1} $ refers to the backwards propagating half of $\Lambda_{D_1}$. 
\begin{lemma} 
The term (\ref{cicada3}) lies in the operator class $I^{{\color{black} 1/4, -1/2 }}(C_{D_0}, \Lambda_{D_1} \circ C_{D_0} ) + I^{{\color{black} 1/4, -1/2 }}(C_{D_0-D_1}, \Lambda_{D_1} \circ C_{D_0-D_1} ) $. 
\end{lemma}
\begin{proof}
{\color{black}
 First of all, notice that $\int_{-x^0}^{x^0} \cdots dr$ is integration over the boxcar function $B(r/x_0) =  H(r+x^0) - H(r-x^0)$. 
 This function has (unitary) Fourier transform 
 \begin{align}
  \widehat{B( \cdot /x_0) } ( \omega)  = \widehat{H( \cdot+x^0)} - \widehat{H(\cdot-x^0)} = \sqrt{\frac{\pi}{2}} e^{-ix^0 \omega } \left(\frac{1}{i \pi \omega} + \delta(\omega) \right)   - \sqrt{\frac{\pi}{2}} e^{ix^0 \omega } \left(\frac{1}{i \pi \omega} + \delta(\omega) \right) 
 \\
  =  \sqrt{\frac{\pi}{2}} (e^{-ix^0 \omega }- e^{ix^0 \omega }) \frac{1}{i \pi \omega}.
  \end{align}
This is actually smooth since the boxcar function is compactly supported. Indeed, $e^{-ix^0 \omega }- e^{ix^0 \omega } = -2i \sin (x^0 \omega)$ cancels out the blowup in $\frac{1}{i \pi \omega}$. 
Multiplying by $1 = \chi(-2\omega) + (1- \chi(-2\omega) )$, where $\chi \in C^\infty_c(\R)$ is a smooth cut-off that equals $1$ near the origin, splits the Fourier transform into
  \[
   \widehat{B( \cdot /x_0) } ( \omega)   =  \sqrt{\frac{\pi}{2}} (e^{-ix^0 \omega }- e^{ix^0 \omega }) \frac{ 1- \chi(-2\omega)  }{i \pi \omega}-  \sqrt{\frac{2}{\pi}} \sin(x^0 \omega) \frac{ \chi(-2\omega)  }{  \omega} , 
  \]
The latter term is in $C^\infty_c $ as a function of $\omega$. 
Consequently,
\begin{equation}\label{123388}
\int_{-x^0}^{x^0} e^{i\frac{r+x^0}{2}\xi_1}  dr  =  \sqrt{2\pi} e^{i\frac{x^0 \xi_1}{2}}  \widehat{B( \cdot /x_0) } (-\xi_1/2 )  = 2    ( 1- \chi(\xi_1) )   \frac{ e^{ ix^0 \xi_1 }}{i  \xi_1} - 2   (1- \chi(\xi_1) )  \frac{1}{i  \xi_1} + d(x^0,x^1,\xi_1) ,
\end{equation}
where $ d  \in  S^{-\infty} ( \R^2 \times \R)$ is a shorthand for $d=  2 e^{i\frac{x^0 \xi_1}{2}}   \sin(x^0 \xi_1 ) \frac{ \chi(\xi_1)  }{  \xi_1} $. 
By substitution, 
 \begin{align}
 \int_{-x^0}^{x^0} \int_{\R^3} e^{i\frac{r+x^0}{2}\xi_1 + i(x'-\tilde{x}')\cdot \xi'   }  c^j(x,\xi') d\xi' dr    = -2 \int_{\R^3} e^{i \varphi_1 }   \frac{   1- \chi(\xi_1)   }{i  \xi_1} c^j(x,\xi') d\xi'   \label{korp1}  \\
 +2 \int_{\R^3} e^{ i \varphi_2 }   \frac{   1- \chi(\xi_1)   }{i  \xi_1} c^j(x,\xi') d\xi'     + \int_{\R^3} e^{ i\varphi_1 }   d(x^0,x^1,\xi_1) c^j(x,\xi') d\xi'    \label{korp2}
 \end{align}
where the phases $\varphi_1 = \varphi_1 (x',\tilde{x}', \xi')$  and $\varphi_2 = \varphi_1 (x^0,x',\tilde{x}', \xi')$ are as in \eqref{varwp} and \eqref{varwp1}.
In coordinates $z = (z',z'',z''') $ such that $z'= (x^2 -\tilde{x}^2 , x^3 - \tilde{x}^3)$ and $z'' = x^1- \tilde{x}^1$, the first integral on the right hand side takes the form
 \begin{align}
  \int_{\R^3} e^{iz' \cdot \zeta' +   z'' \zeta'' }    \frac{   1- \chi(\zeta'')   }{ i\zeta''} \tilde{c}^j(z ,\zeta',\zeta'') d\zeta_1    d\zeta_2    d\zeta_3     \label{modelo}
 \end{align}
 where $\tilde{c}^j \in S^{1}$. The amplitude $  \frac{   1- \chi(\zeta'')   }{ i\zeta''} \tilde{c}^j(z ,\zeta',\zeta'') $ lies in the class $S^{1,-1}$. 
 Indeed, it obeys the estimate
 \begin{align}
 \left| \partial_{z}^\nu \partial_{\zeta''}^{\mu} \partial_{\zeta'}^\rho   \left(   \frac{   1- \chi(\zeta'')   }{ i\zeta''} \tilde{c}^j(z ,\zeta',\zeta'') \right)   \right| 
 \leq\sum_{ \mu_1+ \mu_2 = \mu }   \left| \partial_{\zeta''}^{\mu_1}   \left(   \frac{   1- \chi(\zeta'')   }{ i\zeta''}  \right) \right|  \left|  \partial_{z}^\nu  \partial_{\zeta''}^{\mu_2} \partial_{\zeta'}^\rho  \tilde{c}^j(z ,\zeta',\zeta'')   \right|
 \\
 \leq C \langle \zeta''\rangle^{-1-|\mu_1| } \langle \zeta',\zeta'' \rangle^{1-|\rho|- |\mu_2|} \leq  C \langle \zeta''\rangle^{-1-|\mu| } \langle \zeta',\zeta'' \rangle^{1-|\rho|}
 \end{align}
 in a compact set. 
 Thus, the integral \eqref{modelo} is of the model form for an element in $I^{1,-1} ( S_1 ,S_2)  $, where $S_2 := \{ z' = 0 = z'' \} \subset S_1 := \{ z' = 0 \} $. 
 As shown in \cite[eq. below (1.4)]{Greenleaf-Uhlmann}, the space $I^{1,-1} ( S_1 ,S_2)$ equals $I^{1/4, -1/2} (N^*S_2,N^*S_1)$ which 
 is is just $I^{1/4, -1/2 }(C_{D_0}, \Lambda_{D_1} \circ C_{D_0} )$ in different coordinates. 
Applying an analogous argument, one shows that the second term on the right hand side in (\ref{korp1}-\ref{korp2}) lies in $I^{ 1/4, -1/2 }(C_{D_0-D_1}, \Lambda_{D_1} \circ C_{D_0-D_1} )$. 
The last term $ \int_{\R^3} e^{ i\varphi_1 }   d(x^0,x^1,\xi_1) c^j(x,\xi') d\xi' $ lies in the standard Hörmander class $I^{-1/4} ( C_{D_0})  \subset I^{1/4, -1/2 }(C_{D_0}, \Lambda_{D_1} \circ C_{D_0} )$.
Indeed, $d \in S^{-\infty} ( \R^2 \times \R)$ and $c \in S^1( \R^4 \times \R^3)$, so the product $  d(x^0,x^1,\xi_1) c^j(x,\xi')$ is an element in $S^1( \R^4 \times \R^3)$. 
This finishes the proof. 
}
\end{proof}

We shall now derive the formula \eqref{cano2} regarding $\Lambda_{D_1} \circ  C_{D_0}$ and $\Lambda_{D_1} \circ  C_{D_0-D_1}$ in Boyer-Lindquist coordinates. 
Writing $C_{D_0}$ and $C_{D_0-D_1}$ in these coordinates shows that they are characteristic flowout canonical relations generated by 
\[
\xi_0 =  \eta_1 +\eta_2  =  f^+ P^+_0 + f^- P^-_0 = (f^++ f^-) (p_t + \Psi) + (r-r_s/2) (f^+-  f^-) \sqrt{\Phi}
\]
 and 
\[
\xi_0 - \xi_1 = \eta_1 + 2\eta_2  =f^+ P_0^+ + 2 f^-P^-_0 = (f^++ 2 f^-)  (p_t + \Psi)+ (r-r_s/2) (f^+- 2 f^-) \sqrt{\Phi},
\]
respectively. Above we have omitted the coordinate transformations. Moreover, $\Lambda_{D_1} = \Lambda_{-D_1}$ is the characteristic flowout canonical relation generated by 
\[
-\xi_1 = \eta_2 = f^- P^-_0 =  f^- (p_t + \Psi) +f^- (r-r_s/2)  \sqrt{\Phi}.  %= f^- (p_t + \Psi) - f^-(r-r_s/2) \sqrt{\Phi}
\]
%The non-vanishing coefficient $f^-$ can be divided. Hence, $C_{D_1}^-$ coincides with the characteristic forwards flowout canonical relation $C_{P^-_0}^+$ generated by $P^-_0$. 
%
Each one of the relations above are generated by a function of the form $\mathcal{H} :=  (p_t + \Psi) + \alpha (r-r_s/2)  \sqrt{\Phi} $, where $\alpha$ is non-vanishing.\footnote{In fact, $\alpha =1 $ for the last one. This implies that $\Lambda_{D_1} = \Lambda_{P^-_0} $.}
Indeed, we may assume that $f^\pm$ are both positive and then divide with the coefficient in front of $(p_t + \Psi)$. This does not change the associated canonical relation. 
Let us now compute the Hamiltonian vector field generated by $\mathcal{H}$. It suffices to focus on $\Sigma_2$ as 
\begin{align}
\Lambda_{D_1} \circ  C_{D_0} \subset \{ \xi_1 = 0 = \xi_0 \}^2  \subset \{   f^- P^-_0  =  0, \  f^+ P^+_0 + f^- P^-_0 = 0\}^2 = \{ P^-_0  = 0 = P^+_0\}^2 = \Sigma_2, \\
\Lambda_{D_1} \circ  C_{D_0-D_1} \subset \{ \xi_1 = 0 = \xi_0-\xi_1 \}^2  \subset \{   f^- P^-_0  =  0, \  f^+ P^+_0 + 2f^- P^-_0 = 0\}^2 = \{ P^-_0  = 0 = P^+_0\}^2 = \Sigma_2. 
\end{align} 
A straightforward computation gives 
% the parametrix construction is well understood outside it. 
%\begin{lemma}
%asd
%\end{lemma}

%\begin{proof}
%The two edges of the wedge (corresponding to $r= \tilde{x}_0 - x_0$ and $r = x_0 -\tilde{x}_0 $) are the characteristic flowouts of $f^+ P_0^+ + f^- P_0^-$ and $f^- P_0^-$ along $\Sigma_2$. 
%Each pair in the wedge is connected to each other by a path that combines integral curves of these flows. 
%Let us solve these flowouts. 
%The associated Hamiltonian vector fields are
\begin{align}
\{\mathcal{H}, \cdot \} |_{\Sigma_2}   = \left( \frac{\partial}{\partial t} + \frac{\partial \Psi }{\partial \mo_\phi } \frac{\partial}{\partial \phi} - \frac{\partial \Psi }{\partial r } \frac{\partial}{\partial \mo_r}- \frac{\partial \Psi }{\partial \theta } \frac{\partial}{\partial \mo_\theta}  \right) \Big|_{\Sigma_2}  +  \alpha \sqrt{\Phi} \frac{\partial}{\partial \mo_r}  \Big|_{\Sigma_2}   
=  \left( \frac{\partial}{\partial t} + \frac{c }{ r_s } \frac{\partial}{\partial \phi} +h \frac{\partial}{\partial \mo_r}  \right)  \Big|_{\Sigma_2} ,% \quad h  :=  \left(  \frac{2c \mo_\phi}{r_s^2 }    +  \frac{f^--f^+}{f^++f^-}  \sqrt{\Phi} \right) 
\end{align}
%and
%\[
%\{ f^- P_0^-, \cdot \} |_{\Sigma_2}  =  f^-   \left( \frac{\partial}{\partial t} + \frac{c }{ r_s } \frac{\partial}{\partial \phi} +h_2  \frac{\partial}{\partial \mo_r}  \right)  \Big|_{\Sigma_2}, \quad h_2 :=  \left(  \frac{2c \mo_\phi}{r_s^2 }    +  \frac{f^-}{f^++f^-}  \sqrt{\Phi} \right) .
%\]
for some function $h$.  
This admits integral curves which (as sets) are  of the form 
\[
\{ (t + s, r_s/2 , \theta , \phi  +  \frac{c  }{ r_s } s  \ ; \  -\frac{c  }{ r_s } \mo_\phi , b (s  ) , \mo_\theta, \mo_\phi) : s \in \R \} , 
\]
for arbitrary $ (t,  \theta,\phi  , \mo_r, \mo_\theta, \mo_\phi) $, where $b(s) = b(s, t,  \theta,\phi  , \mo_r, \mo_\theta, \mo_\phi) $ is some smooth function with $\dot{b} = h$ along the curves. 
The curves with $\xi_0$, $\xi_0 - \xi_1$ and $\xi_1$ as the generating function $\mathcal{H}$ differ only in $b$. 
%We will not compute these functions explicitly. 
By combining the paths above we conclude \eqref{cano2}. This finishes the proof. 

\bibliographystyle{plain}
\bibliography{grabib1_ref}

\end{document}